\documentclass[mnsc,nonblindrev]{informs3}

\RequirePackage{tgtermes}
\RequirePackage{newtxtext}
\RequirePackage{newtxmath}
\RequirePackage{endnotes}
\usepackage{changepage}

\OneAndAHalfSpacedXI 

\usepackage{tikz}
\usepackage{algorithm, algpseudocode}

\usepackage{natbib}
 \bibpunct[, ]{(}{)}{,}{a}{}{,}%
 \def\bibfont{\small}%
\usepackage{enumitem}  
\usepackage{dsfont,booktabs}
\usepackage{tikz}
\usetikzlibrary{arrows}
\usetikzlibrary{shapes}
\usetikzlibrary{backgrounds}

\usepackage{caption}
\usepackage{subcaption}
\usepackage{graphicx}

\usepackage{array}
\usepackage{wrapfig}
\usepackage{cite}
\usepackage{bbm}
\usepackage{comment} 
\usepackage{xspace}
\usepackage{algpseudocode}

\usepackage{thmtools} 
\usepackage{thm-restate}

\usepackage[suppress]{color-edits}
\addauthor{sb}{cyan}
\addauthor{ch}{blue}
\addauthor{srs}{orange}
\addauthor{co}{purple}
\addauthor[\textbf{TODO}]{todo}{red}

\usepackage[colorlinks=true,breaklinks=true,bookmarks=true,urlcolor=magenta,citecolor=magenta,linkcolor=magenta,bookmarksopen=false,draft=false]{hyperref}
\usepackage[capitalize]{cleveref}

\usetikzlibrary{arrows.meta, automata,
                positioning,
                quotes}

\ifdefined \argmin
\else \DeclareMathOperator*{\argmin}{arg\,min}
\fi
\ifdefined \argmax
\else \DeclareMathOperator*{\argmax}{arg\,max}
\fi

\newcommand{\norm}[1]{\left \lVert #1 \right \rVert}

\newcommand{\N}{\mathbb{N}}

\newcommand{\R}{\mathbb{R}}

\newcommand{\F}{\mathcal{F}}

\newcommand{\E}{\mathbb{E}}
\newcommand{\pr}{\mathbb{P}}
\renewcommand{\Pr}{\pr}

\newcommand{\NDist}{\mathcal{N}}
\newcommand{\BDist}{\mathcal{B}}

\newcommand{\policy}{\mathcal{A}}
\newcommand{\Nmean}{\mu_{\NDist}}
\newcommand{\Bmean}{\mu_{\BDist}}
\newcommand{\Bstd}{\sigma_{\BDist}}
\newcommand{\Nstd}{\sigma_{\NDist}}
\newcommand{\maxAlloc}{A_{\max}}

\newcommand{\Nmax}{N_{\max}}
\newcommand{\Bmax}{B_{\max}}

\newcommand{\ALG}{{Bang-Bang}\xspace}
\newcommand{\Static}{\textsf{Static}\xspace}

\usepackage{mathtools}

\DeclarePairedDelimiterX{\infdivx}[2]{(}{)}{%
  #1\;\delimsize\|\;#2%
}

\newcommand{\ind}[1]{\mathds{1}\left\{#1\right\}}

\newcommand{\Exp}[1]{\mathbb{E} \left[ #1 \right]} 
\newcommand{\Var}[1]{\mathrm{Var} \left[ #1 \right]} 

\newcommand{\Ttaulow}{\widetilde{\tau}_1}
\newcommand{\Ttauhigh}{\widetilde{\tau}_2}

\newcommand{\Ind}[1]{\mathds{1}\left\{ #1 \right\}}
\newcommand{\Deff}{\Delta_{\text{Efficiency}}\xspace}
\newcommand{\Dfair}{\Delta_{\text{Fair}}\xspace}

\newcommand{\Norm}{\textup{Normal}}
\newcommand{\Exponential}{\textup{Exponential}}
\newcommand{\Poisson}{\textup{Poisson}}
\newcommand{\EG}{\textsf{EG}}
\newcommand{\Yplusk}{Y_+^{(k)}}
\newcommand{\Yminusk}{Y_{-}^{(k)}}
\newcommand{\Bk}{B^{(k)}}

\DeclarePairedDelimiter{\floor}{\lfloor}{\rfloor}

\renewenvironment{proof}[1][Proof]{%
  \par\noindent{\itshape #1.} \ignorespaces
}{%
  \hfill\Halmos\par
}

\renewcommand{\paragraph}[1]{\subsubsection*{#1}}

\EquationsNumberedThrough    

\TheoremsNumberedThrough     
\ECRepeatTheorems  %

\MANUSCRIPTNO{MNSC-0001-2024.00}

\begin{document}


\RUNAUTHOR{Onyeze et al.}

\RUNTITLE{Sequential Fair Allocation With Replenishments}

\TITLE{Sequential Fair Allocation With Replenishments:\\ A Little Envy
Goes An Exponentially Long Way}

\ARTICLEAUTHORS{%

\AUTHOR{Chido Onyeze}
\AFF{School of Computer Science,
Cornell University, \EMAIL{aco59@cornell.edu}}

\AUTHOR{Sean R. Sinclair}
\AFF{Department of Industrial Engineering and Management Sciences,
Northwestern University, \EMAIL{sean.sinclair@northwestern.edu}}

\AUTHOR{Chamsi Hssaine}
\AFF{Department of Data Sciences and Operations, University of Southern California, Marshall School of Business, \EMAIL{hssaine@usc.edu}}

\AUTHOR{Siddhartha Banerjee}
\AFF{School of Operations Research and Information Engineering,
Cornell University, \EMAIL{sbanerjee@cornell.edu}}

} 

\ABSTRACT{%
We study the trade-off between envy and inefficiency in repeated resource allocation settings with stochastic replenishments, motivated by real-world systems such as food banks and medical supply chains. Specifically, we consider a model in which a decision-maker faced with stochastic demand and resource donations must trade off between an {\it equitable} and {\it efficient} allocation of resources over an infinite horizon. The decision-maker has access to storage with fixed capacity $M$, and incurs efficiency losses when storage is empty (stockouts) or full (overflows). We provide a nearly tight (up to constant factors) characterization of achievable envy-inefficiency pairs. Namely, we introduce a class of Bang-Bang control policies whose inefficiency exhibits a sharp phase transition, dropping from $\Theta(1/M)$ when $\Delta = 0$ to $e^{-\Omega(\Delta M)}$ when $\Delta > 0$, where $\Delta$ is used to denote the target envy of the policy. We complement this with matching lower bounds, demonstrating that the trade-off is driven by {\it supply}, as opposed to {\it demand} uncertainty. Our results demonstrate that envy-inefficiency trade-offs not only persist in settings with dynamic replenishment, but are shaped by the decision-maker's available capacity, and are therefore qualitatively different compared to previously studied settings with fixed supply.
}%





\maketitle


\section{Introduction}
\label{sec:intro}

In many shared resource allocation settings --- food banks, power systems, medical supply chains, donation centers, etc. --- a decision-maker faced with stochastic supply and demand seeks to achieve two objectives: allocating as much supply as she has on-hand, while also serving the agents sharing the resource as equitably as possible (equivalently, minimizing the {\it envy} agents incur as a result of inequitable allocations).  In attempting to achieve these objectives, the decision-maker is faced with two sources of inefficiency: stockouts in periods where demand exceeds supply, and overflow in periods where supply exceeds demand. For example, soon after the start of COVID lockdowns, Feeding America reported that $98\%$ of food banks experienced an increase in demand for food assistance, with
$59\%$ of these having low inventory levels~\citep{feedingamerica2020}.
More recently, in the aftermath of the January 2025 wildfires in Los Angeles, 
the overwhelming influx of clothing donations impeded relief efforts due to inventory overflow~\citep{guardian2025wildfires}. 
These high-profile examples reflect the persistent challenges faced in shared resource environments, and underscore the need for equitable allocation policies that are robust to these two sources of inefficiency.

The trade-off between envy and inefficiency has been extensively studied of late, with many works characterizing the Pareto frontier between these objectives under various models~\citep{bertsimas2011price,lien2014sequential,donahue2020fairness,manshadi2021fair,sinclair2021sequential,banerjee2024online,liao2022nonstationary,yin2022optimal}. 
Building on this line of work, this paper examines how stochasticity in \emph{supply} (e.g., incoming donations in the social-good applications described above) shapes a capacitated decision-maker's allocation decisions in dynamic settings. In particular, stochastic incoming supply (also referred to as {\it inventory replenishments}) introduces an important new consideration to existing studies of envy-inefficiency trade-offs: while the decision-maker must allocate conservatively enough to prevent stock-outs, she must also allocate aggressively enough to avoid inventory overflow due to the accumulation of incoming inventory. As a result, perfectly equitable solutions that are not adaptive to current inventory levels are bound to be inefficient, either via frequent stockouts or overflows.

Against this backdrop, this work seeks to answer the following research questions:
\begin{center}
{\em 
Do established envy-inefficiency trade-offs in dynamic environments persist in the presence of stochastic supply with capacity constraints? Can we design policies that achieve the optimal trade-off?
}
\end{center}

\subsection{Main Contributions}
\label{ssec:setting}

We consider a model in which a decision-maker allocates divisible resources between agents arriving online over an infinite horizon. The decision-maker has access to a storage of size $M$. We first consider a setting with a single resource and homogeneous agents. At the beginning of period $t$, the decision-maker receives a stochastic resource donation. She then observes a random number of new agents, each requesting a share of the resource, and chooses how much to allocate to each agent. If the total amount allocated to agents exceeds the amount of available inventory (an event we refer to as a {\it stockout}), we assume the decision-maker covers the additional amount required using an outside option (at a cost). On the other hand, if the amount of remaining inventory exceeds the available storage capacity $M$, any overflow must also be thrown out, similarly at a cost.

The decision-maker evaluates the quality of her policy according to the following three metrics: (i) ex-post envy $\Dfair$, which represents the difference between the maximum and minimum allocations across all rounds, (ii) the overflow (or waste) $\overline{W}$, which represents the long-run average amount of resource thrown out at the end of each period due to overflow, and (iii) the stockout rate $\overline{V}$, representing the long-run average amount of resource that needs to be acquired from an outside option due to stockouts. Given $\overline{W}$ and $\overline{V}$, we define a combined inefficiency metric $\Deff = h\overline{W}+b\overline{V}$, where $h,b$ respectively denote the per-unit overflow/stockout costs.

\begin{figure}[t]
  \centering
    \includegraphics[scale=0.4]{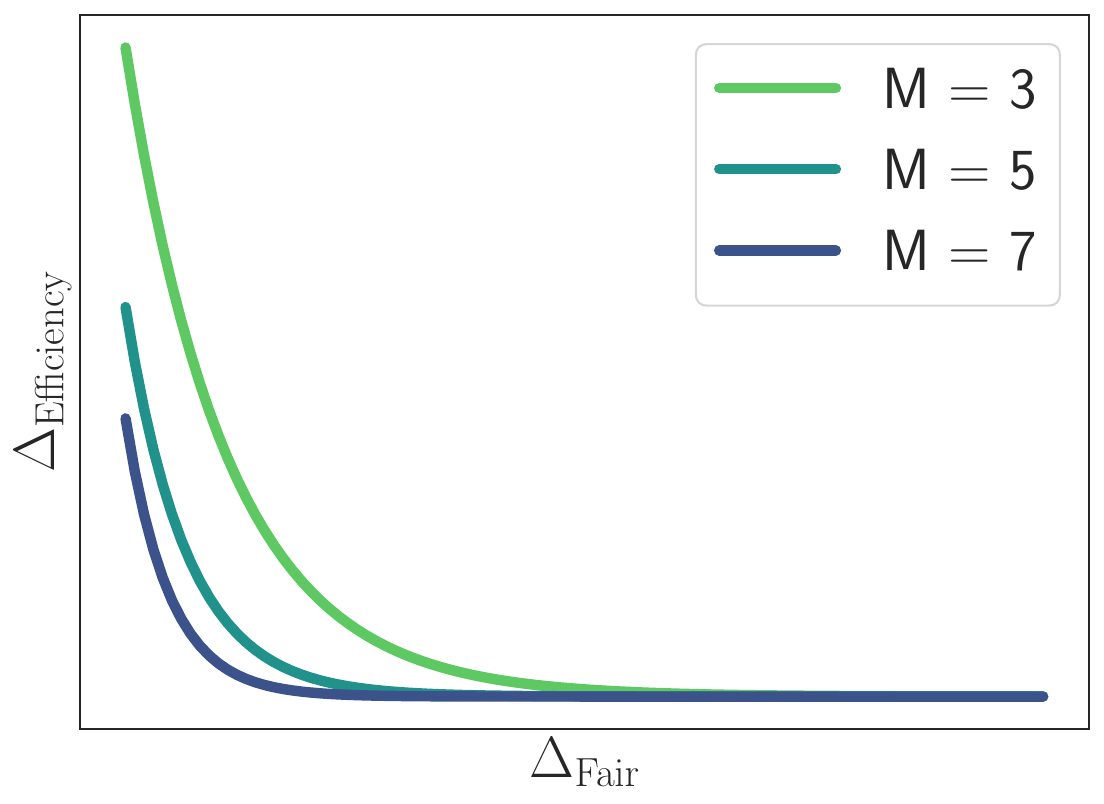}
    \caption{$\Deff$ vs $\Dfair$ for various values of $M$. Each curve illustrates the theoretical lower bound on the achievable inefficiency as a function of $\Dfair$. Even modest increases in $\Dfair$ lead to exponential reductions in inefficiency. Moreover, for fixed $\Dfair$, increasing capacity (darker lines) improves efficiency.}
    \label{fig:intro}
\end{figure}

Our main contribution is to demonstrate that the trade-off between envy and inefficiency exhibits a {\it sharp phase transition} in capacitated settings with stochastic replenishments. In particular:
\begin{enumerate}[leftmargin=0.5cm]
\item We first demonstrate the inefficiency of static allocation policies by showing that the natural \emph{proportional allocation rule} that allocates {the expected donation per agent in each period} (thereby incurring $\Dfair = 0$) leads to $\Deff$ scaling as {$\Theta(M^{-1})$} (\cref{thm:Inefficient_for_fix_0_mean}). This is furthermore the best scaling for {\it any} static allocation, with all other static allocation policies incurring \mbox{$\Deff = \Theta(1)$} (\Cref{thm:Inefficient_for_fix_non-0_mean}).
\item We subsequently show that introducing a limited amount of envy generates {\it exponential} gains in efficiency. Namely, we propose a simple Bang-Bang policy parameterized by $\Dfair > 0$, which over- \mbox{(resp., under-) allocates by $\frac{\Dfair}{2}$, depending on the current inventory level}. We prove that this algorithm achieves
$\Deff = e^{-\Omega(\Dfair M)}$ (\cref{thm:bang-bang_inefficiency_bound}).
\item We prove that the above trade-off is optimal in terms of scaling
(\cref{thm:main_lower_bound}):
even under deterministic demand, any policy that guarantees $\Dfair>0$ must incur {$\Deff =  e^{-O(\Dfair M)}$}. 
\end{enumerate}
Finally, we show that our results naturally extend to settings with multiple resources and heterogeneous agents, thereby demonstrating the universality of our scaling laws.

From a managerial perspective, our results (illustrated in~\Cref{fig:intro}) underscore the central role of capacity constraints in shaping fairness-efficiency trade-offs in real-world settings. In particular, not only do larger capacities directly mitigate inefficiency in perfectly fair settings (since it naturally takes longer for inventory levels to deplete, under a fixed allocation), but they also magnify the efficiency gains achievable from modest relaxations in fairness. These findings highlight the importance of explicitly incorporating capacity considerations into the design of fair allocation policies.

\paragraph{Paper organization.} We next survey the related literature.  In \cref{sec:model} we introduce our basic model, and provide intuition for our results and associated technical challenges in \Cref{ssec:warmup}.  
\cref{sec:achievability} presents our main technical results.  We demonstrate the inefficiency of static allocation policies in \Cref{ssec:fixed_policy}, propose and analyze our Bang-Bang policy in \cref{ssec:Bang-Bang_Policy}. We moreover establish its optimality by providing lower bounds on the envy-inefficiency trade-off in \cref{sec:impossibility}. We extend the model to multiple resources and heterogeneous agents in~\cref{sec:extensions}. Finally, we complement our theoretical results with numerical experiments in \cref{sec:simulations}, and conclude in \Cref{sec:conclusion}.
\subsection{Related Work}
\label{ssec:relwork}

The study of fair resource allocation has deep roots in economics and operations research~\citep{varian1973equity,rawls2017theory,scarf1957min}. We focus here on fair allocation in dynamic settings; see \citet{aleksandrov2019online} for a comprehensive survey of other settings.

\paragraph{Fair resource allocation.} Dynamic fair allocation typically falls into two categories: (i) a fixed set of agents with online resource arrivals~\citep{aleksandrov2015online,mattei2018fairness,gorokh2020fair,bansal2020online,bogomolnaia2022fair,ijcai2019-49,aziz2016control,zeng2019fairness,friedman_2017}, and (ii) a fixed resource budget with agents arriving online. Our work fits into the latter stream, where past works have considered alternative fairness objectives, such as Nash Social Welfare~\citep{azar2010allocate,gorokh2020fair}, max-min fairness~\citep{lien2014sequential,manshadi2021fair}, and variants of envy-freeness~\citep{friedman_2017,cole_2013,10.1145/2764468.2764495,freund2025fair}. Closest to our work are \citet{sinclair2021sequential}, which studied envy-inefficiency trade-offs under a fixed budget, and \citet{banerjee2024online}, which extended this setting to perishable resources.
Similar to these works, our goal is to characterize the trade-off between envy and inefficiency. However, our setting differs fundamentally due to stochastic supply, which causes the decision-maker's capacity constraints to become central to the envy-inefficiency trade-off. Moreover, stochastic replenishments align our model with classical inventory frameworks, naturally leading to an infinite-horizon formulation. As a result, our algorithmic and analytical tools diverge significantly from prior work.

\paragraph{Trade-offs in queueing control.}
Our model builds on the design of fair control policies in queuing systems, beginning with the foundational work of \citet{kelly1998rate}. Much of this literature focuses on \emph{proportional fairness} and, more generally, $\alpha$-fair objectives that balance equity and efficiency \citep{abbas2015fairness,raz2004resource}. Our emphasis on explicit trade-offs is also related to staffing and capacity-sizing results in the Halfin–Whitt regime \citep{halfin1981heavy}.  There, many-server scaling balances two opposing performance metrics: server idleness (inefficiency) and customer delay (a fairness-related service disparity).  In contrast, we study a fixed-capacity (finite-buffer) setting where inefficiency is determined when the inventory level hits the boundaries $\{0,M\}$.  This also has connections to large-buffer asymptotics and large deviations~\citep{ganeshbig}. In our analysis, we show that the Bang-Bang policy induces a ``rate-function'' for the stationary distribution at $\{0,M\}$ that decays as $e^{-\Theta(\Delta M)}$.

A related thread is the quality-of-service (QoS) control literature, which designs scheduling and routing policies to meet deadlines, while ensuring throughput and stability. Foundational work integrates delay, throughput, and reliability constraints~\citep{hou2009theory}, followed by extensions to multi-hop networks with end-to-end deadlines, and predictive scheduling \citep{singh2018throughput,chen2018timely}. We view this line of work as complementary: whereas QoS models trade deadline/latency against throughput and reliability, our model trades ex-post envy against boundary inefficiency under fixed capacity and no waiting.
\section{Preliminaries}
\label{sec:model}

In this section we present our basic model, focusing on the setting with a single resource and homogeneous agents, deferring the extension to multiple resources and heterogeneous agents in \Cref{sec:extensions}.

\paragraph{Model Primitives.} We consider a decision-maker who manages a warehouse (also referred to as a {\em store}) of a single divisible resource over an infinite horizon of discrete time periods.  The store has a capacity $M > 0$ and initially contains $S_0 \in [0,M]$ units of the resource. At the start of period $t \in \N$, the decision-maker receives an additional budget of $B_t \in \R_{\geq 0}$ resources (also referred to as a {\em donation}); moreover, $N_t\in \N$ homogeneous agents arrive, each requesting a share of the resource.
The decision-maker must then decide on a per-agent allocation $A_t \in [0,\maxAlloc]$, where $\maxAlloc \in \mathbb{R}_{> 0}$ is a finite upper bound on the maximum possible allocation in each period.\footnote{The assumption that each individual receives the same allocation $A_t$ follows from the fact that agents are homogeneous~\citep{varian1973equity}.} We let $S_{t-1} \in [0,M]$ be the inventory level at the start of round $t$. When the total allocation exceeds the post-donation inventory level (i.e., $S_{t-1}+B_t < N_tA_t$), the decision-maker supplements the allocation via an outside option (at a cost). On the other hand, if the post-allocation inventory level exceeds the store's capacity $M$, any excess is thrown away. The store's inventory level is updated at the end of round $t$ via the following inductive relation: 
\begin{equation}
\label{eq:state_dynamics}
S_{t} = (S_{t-1} + B_t - N_t A_t)\Big|_0^{M},
\end{equation}
where we use the notation $x|_0^{M} = \max\{0,\min\{x,M\}\}$ to denote the projection of $x$ in $[0,M]$. For ease of notation, we define the {\it inventory drift} in round $t$ to be $Z_t = B_t-N_tA_t$, noting that $S_{t} = \left.(S_{t-1} + Z_t)\right|_0^{M}$, by \Cref{eq:state_dynamics}. For ease of notation, we let $Z_{\min} = \maxAlloc\Nmax$ and $Z_{\max} = \Bmax$, noting that $Z_t \in [-Z_{\min},Z_{\max}]$ for all $t$.

We assume $B_t$ is drawn i.i.d. from a distribution $\BDist$ supported on a subset of $[0, B_{\max}]$, and let $\Bmean {> 0}$ denote the mean of $\BDist$, known to the decision-maker. Similarly, $N_t$ is drawn i.i.d. from a distribution $\NDist$ supported on a subset of $\{0,\ldots,N_{\max}\}$, and let $\Nmean {> 0}$ denote its known mean.\footnote{The assumption that $N_t$ is discrete and both $N_t$ and $B_t$ are bounded is for technical simplicity. Our results extend naturally for any $\NDist$ and $\BDist$ which are sub-Gaussian.} Finally, we assume $\BDist$ and $\NDist$ are independent of $M$, and that $\BDist$ has strictly positive variance $\Bstd^2$.

\begin{remark}
The assumption that the supply and demand processes are i.i.d. across time is for analytical tractability. In Section \ref{sec:experiments_time_varying} we numerically demonstrate that our envy-inefficiency laws continue to hold under seasonal supply and demand patterns that one expects to observe in practice.
\end{remark}

\paragraph{Policies.} We define an allocation policy $\mathcal{A}$ to be a mapping from the current round $t$, the realized supply and demand $B_t$ and $N_t$, and the current inventory level $S_{t-1}$, to a per-individual allocation quantity $A_t$.
Policy $\mathcal{A}$ is said to be {\it time-homogeneous} if 
it only depends on the current round $t$ through the inventory level $S_{t-1}$ and the realized demand and supply. 
Our achievability results all involve time-homogeneous policies, while our impossibilities hold under arbitrary policies.

\paragraph{Measures of Envy and Inefficiency.}  Our aim is to understand the trade-off between {\it envy} (unfairness) and {\it inefficiency}, as described in the introduction. 
To formalize this trade-off, we measure the {\it unfairness} of any policy $\mathcal{A}$ via its \emph{ex-post} --- or {\it hindsight} --- {\it envy}, in line with prior literature~\citep{varian1973equity,sinclair2021sequential,banerjee2024online}.
\begin{definition}[Ex-post envy]
\label{def:fairness}
The \emph{ex-post envy} of a sequence of allocations $(A_t, t\in \N)$ is given by:
\begin{equation}
\Dfair \;=\; \sup_{t,t': N_t, N_{t'} > 0} \, |A_t - A_{t’}|.
\end{equation}
\end{definition}

This metric is deliberately stringent: it bounds the disparity between any two agents \emph{anywhere} along the sample path, rather than on average, or with time-discounting. We choose this worst-case metric since average-case metrics can obscure persistent disparities or biases in allocations.  For instance, a food bank that distributes equal amounts on most days but shuts down entirely on specific days may exhibit low average envy, but still leaves certain populations (such as those only available on closure days) underserved.

This standard metric of unfairness also helps to motivate the necessity of supplementing (at cost) the allocation in periods of stockout. It is easy to see that if the algorithm is restricted to not allocate anything in periods where the inventory level is zero, it necessarily has $\Dfair = \Omega(1)$; such events are likely to occur over sample paths with prolonged spells of zero donations or surge demand.
This helps motivate our inefficiency metric, which is in line with standard models of inventory control~\citep{scarf1957min}.
In particular, we define the {\it overflow} (or {\it waste}) $W_t$ incurred in round $t$ as the amount of resources lost due to the inability to store more resources than capacity $M$. Formally:
\begin{equation}\label{eq:waste}W_t = (S_{t-1} + B_t - N_t A_t - M)^+,
\end{equation}
where $(\cdot)^+ = \max\{\cdot, 0\}$ is used to define the positive part. We moreover define the {\it stockout quantity} $V_t$ in round $t$ to be the additional resource the decision-maker needs to acquire in order to allocate $A_t$. Formally:
\begin{equation}\label{eq:subsidy}V_t = (S_{t-1}+B_t-N_t A_t)^-,\end{equation}
where $(\cdot)^- = -\min\{\cdot, 0\}$.
With these definitions in hand, we can define our measure of inefficiency.
\begin{definition}[Inefficiency]
\label{def:efficiency}
Given per-unit overflow cost $h \in \mathbb{R}_{> 0}$ and per-unit stockout cost $b \in \mathbb{R}_{> 0}$, the \emph{inefficiency} of a sequence of allocation decisions $(A_t, t\in \N)$ is given by:
\begin{equation}
\label{eq:deff}
\Deff = \limsup_{T \rightarrow \infty} \left( \frac{1}{T} \sum_{t=1}^T hW_t + bV_t \right) =  h\overline{W} + b\overline{V},
\end{equation}
where $\overline{W}$ and $\overline{V}$ respectively denote the long-run average overflow and stockouts under $(A_t, t \in \mathbb{N})$, i.e.:
\begin{equation*}
\overline{W} = \limsup_{T \rightarrow \infty} \frac{1}{T} \sum_{t=1}^T W_t, \quad \overline{V} = \limsup_{T \rightarrow \infty} \frac{1}{T} \sum_{t=1}^T V_t.
\end{equation*}
\end{definition}

\begin{remark}
While our measure of inefficiency is natural for the applications motivating our work, our analysis provides explicit bounds on the inventory level's hitting times of states $\{0,M\}$, as well as the stationary probabilities of these states. Thus, a corollary of our analysis is a scaling law for {\it any} inefficiency metric that is a function of these states.
\end{remark}

\paragraph{Envy-Inefficiency Trade-off.}  We say that a policy achieves the \emph{optimal envy-inefficiency trade-off} if it lies on the Pareto frontier of the two metrics $(\Deff,\Dfair)$, 
defined as all pairs $(\Deff,\Dfair)$ that are the solution to the following optimization problem:
\begin{equation}
\label{eq:goal}
    \min_{A_t, t \in \mathbb{N}} \Deff \quad \text{s.t. } \quad \Dfair \leq \Delta \,\text{ almost surely},
\end{equation}
where $\Delta \geq 0$ can be thought of as an ``unfairness budget'' given to the allocation policy. It is easy to see that, for any $\Delta \geq 0$, there exists an optimal solution to \Cref{eq:goal} such that $\Dfair = \Delta$. As a result, in the remainder of this work we use $\Delta$ and $\Dfair$ interchangeably. 

Note that a trivial part of the Pareto frontier is obtained by considering $\Delta \rightarrow +\infty$, in which case one optimal policy is to simply set $A_t = (S_{t-1}+B_t) / N_t$ (i.e., to allocate the entirety of the available inventory) whenever $N_t>0$.
This results in the minimum possible $\Deff \approx 0$, where the inefficiency corresponds to the average amount by which the total incoming budget exceeds $M$ over a stretch of periods with zero demand. Our focus is thus on the scaling behavior of $\Deff$ as a function of $M$ and $\Delta$.
\section{Envy-Inefficiency Scaling Laws: Intuition and Technical Challenges}
\label{ssec:warmup}

Our key observation in the derivation of our envy-inefficiency scaling laws is that the long-run average overflow (respectively, stockout quantity) is proportional to the long-run average fraction of time for which $S_t = M$ (resp., $S_t = 0$). Therefore, it suffices to characterize the steady-state probability that $S_t \in \{0,M\}$. In general, however, $S_t$ is a complex Markov chain over a continuous state space; consequently, there are few standard tools for characterizing its steady-state measure. In this section, we build some intuition into the form of our final envy-inefficiency scaling laws by considering more tractable processes that approximate the inventory dynamics of our real system at a high level. We then preview the technical challenges in translating these insights to the real dynamics.

\paragraph{Understanding achievable scaling laws via a birth-death process.} 
We begin by illustrating our bounds on $\Deff$ via a stylized birth–death process, where a birth represents the incoming budget exceeding the allocated quantity, and a death represents the case where the allocation exceeds the incoming budget. While this simplified model does not directly correspond to the inventory dynamics of our system, it captures some of its key structural features. Hence, one would expect a similar scaling. In this section we overload notation and let $S_t$ be this stylized birth-death process, with discrete state space $\{0, 1, \ldots, M\}$. Moreover, $Z_t \in \{-1,0,1\}$ is now used to denote the drift of this birth-death process.

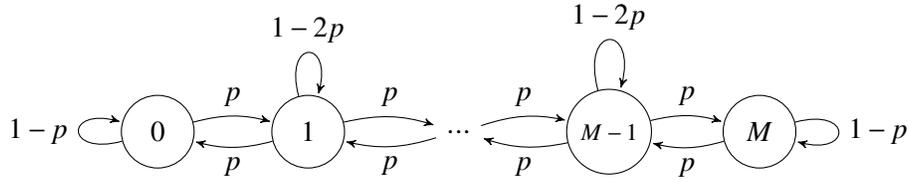
\begin{figure}[!t]
    \centering
\scalebox{1}{
    \begin{tikzpicture}[->,>=stealth',shorten >=1pt]
    \node[state] (1) at (0,0) {$0$};
    \node[state] (2) at (2,0) {$1$};
    \node (3) at (4,0) {$...$};
    \node[state] (4) at (6,0) {\footnotesize $M-1$};
    \node[state] (5) at (8,0) {$M$};
        \draw (1) edge[bend left=15] node[above] {$p$} (2);
        \draw (1) edge[loop left] node{$1-p$} (1);
        \draw (2) edge[bend left=15] node[below] {$p$} (1);
        \draw (2) edge[bend left=15] node[above] {$p$} (3);
        \draw (3) edge[bend left=15] node[below] {$p$} (2);
        \draw (3) edge[bend left=15] node[above] {$p$} (4);
        \draw (4) edge[bend left=15] node[below] {$p$} (3);
        \draw (4) edge[bend left=15] node[above] {$p$} (5);
        \draw (5) edge[bend left=15] node[below] {$p$} (4);
        \draw (5) edge[loop right] node{$1-p$} (5);
        \draw (2) edge[loop above] node{$1-2p$} (2);
        \draw (4) edge[loop above] node{$1-2p$} (4);
\end{tikzpicture}}
    \caption{Markov chain representation of the birth-death process $S_t$ for $\Dfair = 0$. 
    }
    \label{fig:gamblers ruin chain}
\end{figure}

Fix $p \in (0,1/2)$, and consider the birth-death process  shown in~\Cref{fig:gamblers ruin chain}, formally defined by:
\begin{align*}
\begin{cases}
\pr(Z_t = 1 \mid S_{t-1}) = \pr(Z_t = -1 \mid S_{t-1}) = p \qquad &\text{if } S_{t-1} \in \{1,2,\ldots, M-1\}\\
\pr(Z_t=1 \mid S_{t-1})=0 \quad &\text{if } {S}_{t-1}=M \\
\pr({Z}_t=-1 \mid S_{t-1})=0 \quad &\text{if } {S}_{t-1}=0.
\end{cases}
\end{align*}

Note that $\Exp{Z_t \mid S_{t-1}} = 0$ whenever $S_{t-1} \in \{1, \ldots, M-1\}$.  This setting can be viewed as approximating the true inventory level under the static allocation rule \mbox{$A_t = \Bmean / \Nmean$} for all $t \in \mathbb{N}$, under which $\mathbb{E}[Z_t \mid S_{t-1}] = 0$ as well, and for which $\Dfair = 0$. It is easy to verify that, for any $p \leq 1/2$, the stationary distribution $\pi$ of the ${S}_t$ birth-death Markov chain is given by $\pi[s] = \frac{1}{M+1}$ for all $s\in\{0,1,\ldots,M\}$. Since the long-run average overflow (respectively, stockout quantity) is proportional to the fraction of time for which ${S}_t=M$ (resp., ${S}_t=0$), we obtain $\Deff = \Theta(1 / M)$ when $\Dfair=0$.

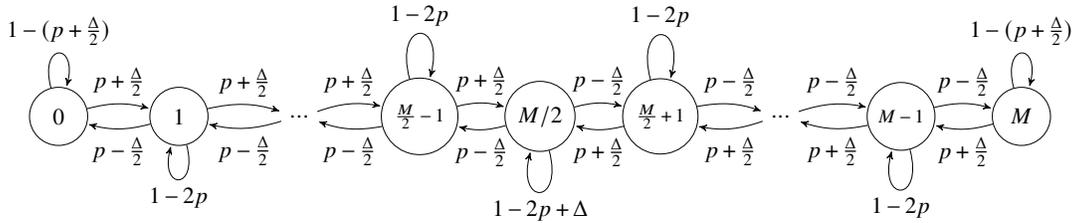
\begin{figure}[!t]
    \centering
\scalebox{.8}{
\begin{tikzpicture}[->,>=stealth',shorten >=1pt]
    \node[state] (1) at (0,0) {$0$};
    \node[state] (2) at (2,0) {$1$};
    \node (3) at (4,0) {$...$};
    \node[state] (4) at (6,0) {\footnotesize $\frac{M}{2}-1$};
    \node[state] (5) at (8,0) {$M/2$};
    \node[state] (6) at (10,0) {\footnotesize $\frac{M}{2}+1$};
    \node (7) at (12,0) {$...$};
    \node[state] (8) at (14,0) {\footnotesize$M-1$};
    \node[state] (9) at (16,0) {$M$};
        \draw (1) edge[bend left=15] node[above] {$p+\frac{\Delta}{2}$} (2);
        \draw (1) edge[loop above] node{$1-(p+\frac{\Delta}{2})$} (1);
        \draw (2) edge[bend left=15] node[below] {$p-\frac{\Delta}{2}$} (1);
        \draw (2) edge[bend left=15] node[above] {$p+\frac{\Delta}{2}$} (3);
        \draw (3) edge[bend left=15] node[below] {$p-\frac{\Delta}{2}$} (2);
        \draw (3) edge[bend left=15] node[above] {$p+\frac{\Delta}{2}$} (4);
        \draw (4) edge[bend left=15] node[below] {$p-\frac{\Delta}{2}$} (3);
        \draw (4) edge[bend left=15] node[above] {$p+\frac{\Delta}{2}$} (5);
        \draw (5) edge[bend left=15] node[below] {$p-\frac{\Delta}{2}$} (4);
        \draw (5) edge[bend left=15] node[above] {$p-\frac{\Delta}{2}$} (6);
        \draw (6) edge[bend left=15] node[below] {$p+\frac{\Delta}{2}$} (5);
        \draw (6) edge[bend left=15] node[above] {$p-\frac{\Delta}{2}$} (7);
        \draw (7) edge[bend left=15] node[below] {$p+\frac{\Delta}{2}$} (6);
        \draw (7) edge[bend left=15] node[above] {$p-\frac{\Delta}{2}$} (8);
        \draw (8) edge[bend left=15] node[below] {$p+\frac{\Delta}{2}$} (7);
        \draw (8) edge[bend left=15] node[above] {$p-\frac{\Delta}{2}$} (9);
        \draw (9) edge[bend left=15] node[below] {$p+\frac{\Delta}{2}$} (8);
        \draw (9) edge[loop above] node{$1-(p+\frac{\Delta}{2})$} (9);
        \draw (2) edge[loop below] node{$1-2p$} (2);
        \draw (4) edge[loop above] node{$1-2p$} (4);
        \draw (5) edge[loop below] node{${1-2p+\Delta}$} (5);
        \draw (6) edge[loop above] node{$1-2p$} (6);
        \draw (8) edge[loop below] node{$1-2p$} (8);
    \end{tikzpicture}
}
    \caption{Markov chain representation of the approximation of the birth-death process ${S}_t$ under a Bang-Bang policy for which $\Dfair = \Delta > 0$.
    }
    \label{fig:biased_gamblers_ruin_chain}
\end{figure}

Consider now the case where $\Dfair \in (0,2p)$. By allowing our policy to incur envy, we can now vary the allocations (and hence the inventory drift) in order to drive the system away from the boundary states $\{0,M\}$.  To do so, we construct a modified birth-death chain, shown in \Cref{fig:biased_gamblers_ruin_chain}. Specifically, for this Markov chain, for $S_{t-1} \in (M/2,M)$:
\begin{align*}
\begin{cases}
\pr({Z}_t=1 \mid S_{t-1})&=p-\frac{\Dfair}{2}\\
\pr({Z}_t=-1 \mid S_{t-1})&=p+\frac{\Dfair}{2}.
\end{cases}
\end{align*}
Moreover, for $S_{t-1} \in (0, M/2)$, we have:
\begin{align*}
\begin{cases}
\pr({Z}_t=1 \mid S_{t-1})&=p+\frac{\Dfair}{2} \\
\pr({Z}_t=-1 \mid S_{t-1})&=p-\frac{\Dfair}{2}.
\end{cases}
\end{align*}
Finally, at $S_{t-1} = M/2$, we have $\mathbb{P}(Z_t = 1 \mid S_{t-1}) = \mathbb{P}(Z_t = -1 \mid S_{t-1}) = p-\frac{\Dfair}{2}$.

Note that $\Exp{Z_t \mid S_{t-1}} = \Dfair > 0$ for $S_{t-1}  \in (0,M/2)$, and $\Exp{Z_t \mid S_{t-1}} = -\Dfair < 0$ {for $S_{t-1} \in (M/2,M)$}, inducing $M/2$-reversion in the drift of the birth-death process.  This setting can be viewed as approximating the true inventory level under a state-dependent allocation rule $A_t$ that perturbs $\Bmean / \Nmean$ by $\pm \Dfair/2$, as our \ALG policy will do in \cref{ssec:Bang-Bang_Policy}.

Solving for the stationary distribution $\pi$ in this case,
we obtain $\pi[0] = \pi[M] = \left(2\cdot\frac{\rho^{M/2}-1}{\rho-1}+\rho^{M/2}\right)^{-1}$, where $\rho = \frac{2p+\Dfair}{2p-\Dfair}$. (We refer the reader to Appendix \ref{mc:deriv} for a derivation of this fact.) Noting that, for constant $p \in (0,1/2)$ and small $\Dfair > 0$, $\rho = e^{\Theta(\Dfair)}$, we obtain \mbox{$\pi[0] = \pi[M] = e^{-\Theta(\Dfair M)}$}. As a result, $\Deff = e^{-\Theta(\Dfair M)}$.

\paragraph{Overview of technical challenges.} While the above heuristic arguments provide some intuition into the form of our envy-inefficiency trade-off, the technical challenge lies in formally establishing these results for the actual inventory dynamics, and under a wide range of demand and supply distributions. We briefly highlight some of the main technical ideas.

The main difficulty in bounding the overflow $\overline{W}$ and stockout quantity $\overline{V}$ under the true inventory dynamics stems from the structure of the process $S_t$ (as defined in~\Cref{eq:state_dynamics}).
This process forms a Markov chain that evolves continuously within the interior $(0, M)$ but has absorbing states $\{0, M\}$, making it challenging to analyze or characterize the steady-state probabilities at these boundaries using standard techniques.
We eschew this difficulty by using a novel renewal-reward argument to express the steady-state mass of states $0$ and $M$ in terms of their hitting times and hitting probabilities. We then derive bounds for these quantities under the static and Bang-Bang policies we consider via a careful martingale analysis.

\section{Envy-Inefficiency Trade-offs under Time-Homogeneous Allocation Policies}
\label{sec:achievability}

As described in \Cref{sec:model}, the overflow and stockout quantities $\overline{W}$ and $\overline{V}$ are the driving measures of inefficiency in our setting. In this section, we develop a general framework for analyzing these for arbitrary time-homogeneous allocation policies, leveraging tools from renewal theory. Building on this, we characterize the extreme point of the envy-inefficiency trade-off when $\Dfair = 0$. We then introduce a simple class of adaptive control policies  parameterized by a target fairness bound $\Delta$, and use our framework to demonstrate that a limited amount of adaptivity generates exponential gains in efficiency. 

\subsection{Bounds on Overflow and Stockout}
\label{ssec:timehomogen}

We first establish the intuitive fact that the long-run average waste and stockout quantity respectively grow with the fraction of time the warehouse is full (i.e., $S_t = M$) and empty (i.e., $S_t = 0$). For ease of notation, we let $H_M = \lim_{T \rightarrow \infty}\frac{1}{T}\sum_{t=1}^T \ind{S_{t} = M}$ and $H_0 = \lim_{T \rightarrow \infty}\frac{1}{T}\sum_{t=1}^T \ind{S_{t} = 0}$ under any policy $\mathcal{A}$. Additionally, for any $\alpha \in [0,A_{\max}]$, we let $Z_t(\alpha) = B_t - \alpha N_t$ be the drift in round $t$ for any policy that allocates $\alpha$ in that round.

\begin{restatable}{proposition}{MainCorr}\label{Cor:Main_expression_for_inefficiency}
Consider any allocation policy such that $A_t = \alpha_M$ if $S_{t-1} = M$, and $A_t = \alpha_0$ if $S_{t-1} = 0$, for some $\alpha_M, \alpha_0 > 0$.  If $H_M$ and $H_0$ exist, then:
\begin{align*}
{\Exp{Z_1(\alpha_M)^+} H_M \leq \overline{W} \leq Z_{\max} \cdot H_M \quad \text{and} \quad \Exp{Z_1(\alpha_0)^-} H_0 \leq \overline{V} \leq Z_{\min} \cdot H_0.}
\end{align*}
\end{restatable}
We provide some intuition for \Cref{Cor:Main_expression_for_inefficiency}, deferring a formal proof to Appendix \ref{apx:aux-res}. The upper bound on $\overline{W}$ follows from the fact that waste is only incurred in periods where $S_t = M$; in these periods, the amount thrown away is the excess over $M$, i.e., $S_{t-1}+B_t-N_tA_t-M \leq Z_{\max}$. The lower bound, on the other hand, follows from the fact that, in periods where $S_{t-1} = M$, the waste incurred is given by $(B_t-N_tA_t)^+ = Z_1(\alpha_M)^+$, by definition. The bounds on $\overline{V}$ follow similar lines. Namely, a stockout only occurs in periods $t$ such that $S_t = 0$; in these periods, the maximum amount by which demand can exceed the available budget is $Z_{\min}$. Moreover, whenever $S_{t-1} = 0$, the policy allocates $\alpha_0$ by assumption, incurring an underage of $Z_1(\alpha_0)^-$.

Both the static and Bang-Bang policies we analyze in the remainder of this section satisfy the condition that a fixed amount is allocated at the boundaries.  Hence, it suffices to characterize $H_M$ and $H_0$ in order to obtain tight bounds on the inefficiency of these policies.
Toward this goal, we introduce the following notation:
\begin{itemize}
    \item For any initial state $S \in [0, M]$, we let $\tau(S) = \inf\left\{t > 0 : S_{t} \in \{0, M\} \mid S_0 = S \right\}$ be the first time $S_t$ hits $0$ or $M$, and define $E(S) = \mathbb{E}[\tau(S)]$. Additionally, let {$\tau_M(S) = \inf\{t > 0 : S_t = M \mid S_0 = S\}$ and $\tau_0(S) = \inf\{t > 0 : S_t = 0 \mid S_0 = S\}$ respectively denote the first time the process reaches $M$ or $0$, starting from state $S$.}
    \item We let $p_M = \pr(S_{\tau(M)} = M \mid S_0 = M)$ be the probability that, given initial state $M$, $S_t$ reaches $M$ before 0. Similarly  $p_0 = \pr(S_{\tau(0)} = 0 \mid S_0 = 0)$ is used to denote the probability that $S_t$ reaches 0 before $M$, given that initial state 0.
\end{itemize}

The following theorem, one of the main drivers of our results, relates $H_M$ and $H_0$ to the hitting times introduced above.

\begin{restatable}{theorem}{InefficiencyExpression} 
\label{Thm:Main_expression_for_inefficiency}
{Fix a time-homogeneous policy $\mathcal{A}$ such that $E(M) < \infty$, $E(0) < \infty$, $p_M \in (0,1)$ $p_0 \in (0,1)$, and $\tau_M(S_0)$ and $\tau_0(S_0)$ are finite almost surely.  Then $H_M$ and $H_0$ exist almost surely, and satisfy:}
\begin{equation}
\label{eq:hit_count}
    H_M = \frac{1}{E(M) + E(0)\cdot \frac{1-p_{M}}{1-p_{0}}},\;\;\;\;\;\;\;\;\;\;\;\;\;\; \text{and} \;\;\;\;\;\;\;\;\;\;\;\;\;\; H_0 = \frac{1 }{E(M)\cdot \frac{1-p_{0}}{1-p_{M}} + E(0)}.
\end{equation} 
\end{restatable}

\Cref{Thm:Main_expression_for_inefficiency} highlights the driving forces behind $H_M$ and $H_0$. Namely, it describes the intuitive phenomenon that policies which take longer to reach either boundary (via an increase in either $E(M)$ or $E(0)$) will spend less time at these boundaries. Additionally, it highlights the saliency of the quantity $\frac{1-p_M}{1-p_0}$. In words, $1-p_M$ corresponds to the probability that $S_t$ reaches 0 before $M$ given initial state $M$; $1-p_0$ corresponds to the probability that $S_t$ reaches $M$ before 0 given initial state 0. For higher values of $\frac{1-p_M}{1-p_0}$, the inventory level drifts away from $M$ and toward 0, resulting in a smaller (resp., larger) fraction of time spent at $M$ (resp., 0). This relationship also underpins our envy-inefficiency scaling laws: policies that slow the drift toward \(0\) or \(M\) (lengthening \(E(0)\) and \(E(M)\)) and that avoid heavily favoring one boundary over the other (letting \(p_0 \approx p_M\)) reduce both stockouts and overflows, at the cost of occasionally larger disparities in per-period allocations.

We defer the full proof of \Cref{Thm:Main_expression_for_inefficiency} to Appendix \ref{apx:h-thm}, and provide a proof sketch below.

\begin{proof}[Proof sketch] The result follows from a careful application of the Renewal Reward theorem~\citep[Chapter 10]{grimmett2020probability}. 
Since $\tau_M(S_0) < \infty$ almost surely, we can assume without loss of generality that the starting state $S_0 = M$.  We prove the result for $H_M$; the argument for $H_0$ is analogous.

A natural way to express $H_M$ would be to decompose the horizon into cycles which start at $S_t = M$ and end at the smallest $t' > t$ for which $S_{t'} = M$. While this can yield succinct expressions for $H_M$, it turns out that the distribution of these cycles is difficult to analyze directly for the allocation policies we study. We overcome this difficulty by decomposing the horizon into inventory-dependent cycles that begin at $M$, hit zero, and return to $M$.
In more detail, we define a sequence of stopping times $0 = \tau_0,\tau_1,\tau_2,\ldots$ such that the $n$-th renewal cycle of the process begins in round $\tau_{n-1}$ and ends in round $\tau_n-1$, with
\[\tau_n = \inf\{t > \tau_{n-1}: S_t = M \text{ and there exists } \tau_{n-1} < t' < t \text{ such that } S_{t'} = 0\}.\]
Note that $S_t$ may hit 0 or $M$ multiple times within a cycle. One such cycle is depicted in \Cref{fig:renewal}.

\begin{figure}[t]
\centering
\scalebox{0.5}{
\input{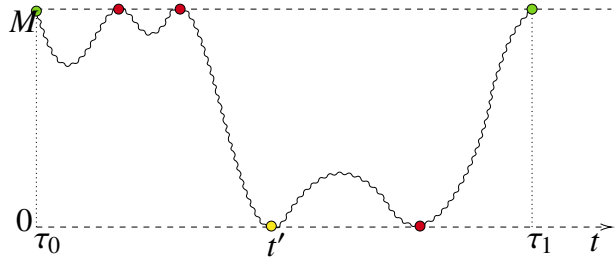}
}
\caption{{Renewal cycle used in proof of \Cref{Thm:Main_expression_for_inefficiency}. The first cycle begins $\tau_0 = 0$ with $S_{\tau_0}=M$; the inventory level hits $0$ for the first time at $t'$, and then returns to $M$ for the first time afterwards at $\tau_1$.}}
\label{fig:renewal}
\vspace{-0.5em}
\end{figure}

Let $L_n = \tau_{n} - \tau_{n-1}$ be the length of the $n$-th renewal cycle, and $R_n = \sum_{t=\tau_{n-1}}^{\tau_n-1} \ind{S_t = M}$ be the number of times in which $S_t = M$ in the cycle. (For instance, $R_1 = 3$ in \Cref{fig:renewal}.) Since $\policy$ is time-homogeneous, 
$(L_n, R_n)$ form a renewal-reward process. By the Renewal Reward Theorem, we then have
\begin{equation}\label{eq:renewal-reward_expression}
    \lim_{T \rightarrow \infty}\frac{\sum_{t=1}^T\mathds{1}\{S_t = M\}}{T} \stackrel{a.s.}{=} \frac{\E[R_1]}{\E[L_1]}.
\end{equation}
Thus, it suffices to compute $\E[R_1]$ and $\E[L_1]$. 

Our first key observation is that $R_1-1$ is geometrically distributed. To see this, consider the counting process that increments each time $S_t \in \{0,M\}$, and define a success to be any trial for which $S_t = 0$, an event which occurs with probability $1-p_M$. Then, $R_1-1$ corresponds to the number of failures before first success, and is therefore geometric with mean $\frac{p_M}{1-p_M}$. Adding back the first time $S_t = M$ (which occurs at the beginning of the cycle, by assumption), we obtain $\E[R_1] = \frac{1}{1-p_M}$.
\end{proof}

\medskip 
Note that $S_t$ is a reflected random walk with barriers at 0 and $M$. In order to bound $E(M), E(0), p_0$ and $p_M$, we consider the ``unreflected'' counterpart of $S_t$, which we denote by $Q_t$ and which is coupled to $S_t$ prior to $\tau(S)$. The following lemma establishes that it suffices to focus on this process in the remainder of our analysis; this will allow us to leverage martingale arguments in bounding these four quantities of interest. We defer its straightforward proof to Appendix \ref{apx:aux-res}.

\begin{restatable}{lemma}{RandomWalkEquivalence}
\label{lem:Random_walk_equivalence_1}
Let $Q_t = \sum_{s=1}^t Z_s + S_0$, and define stopping time \mbox{$\tau = \inf\{t > 0: Q_t \not\in (0, M)\}$}. Then, for any $S \in [0, M]$, $E(S) = \E[\tau \mid S_0 = S]$. Furthermore, \mbox{$p_M = \pr(Q_{\tau} \geq M \mid S_0 = M)$} and $p_0 = \pr(Q_{\tau} \leq 0 \mid S_0 = 0)$.
\end{restatable}

\subsection{Inefficiency of Static Allocation Policies ($\Dfair=0$)} 
\label{ssec:fixed_policy} 

We now characterize the inefficiency of {\it static} policies, wherein $A_t=\alpha$ for all $t \in \mathbb{N}$, for some $\alpha \geq 0$. Note that by definition, $\Dfair = 0$ for this class of policies.

We first turn our attention to the natural static policy that sets $\alpha^* := \frac{\mu_{\BDist}}{\mu_{\NDist}}$ in each period. In the remainder of this work, we refer to this policy as the {\it proportional} allocation policy. \Cref{thm:Inefficient_for_fix_0_mean} recovers the scaling behavior of the stylized birth-death process analyzed in \Cref{ssec:warmup}.

\begin{restatable}{theorem}{InefficiencyZeroMean} \label{thm:Inefficient_for_fix_0_mean}
The static policy that allocates $\alpha^* = \Bmean / \Nmean$ in each round {achieves}: $$\overline{W} = \Theta(M^{-1}) \;\;\;\;\;\;\;\; \text{ and } \;\;\;\;\;\;\;\; \overline{V} = \Theta(M^{-1}).$$
    Therefore, $\Deff = \Theta(M^{-1})$.
\end{restatable}

Note that under this policy, the inventory drift $Z_t(\alpha^*)$ satisfies $\Exp{Z_t(\alpha^*)} = 0$. We will later see that this choice is uniquely optimal across the space of all static policies. Furthermore, \cref{thm:Inefficient_for_fix_0_mean} illustrates the poor performance of static policies; in particular, we will later see that extremely simple adaptive policies achieve an inefficiency scaling as $e^{-\Omega(M \Delta)}$, an exponential improvement over the proportional allocation policy.

We provide a sketch of the proof of \Cref{thm:Inefficient_for_fix_0_mean}, deferring the complete proof to Appendix \ref{apx:thm-2-proof}.

\begin{proof}[Proof sketch]
For ease of notation, we let $Z_t = Z_t(\alpha^*)$ for all $t$, and let $\sigma^2 = \mathbb{E}[Z_1^2]$. We first argue that the stopping times $E(M), E(0)$ and hitting probabilities $p_M, p_0$ are well defined and finite, since $\text{Var}[B_t] > 0$. Thus, by \cref{Cor:Main_expression_for_inefficiency} and \cref{Thm:Main_expression_for_inefficiency}, it suffices to deduce expressions for $E(M), p_M, E(0)$, and $p_0$. 

We then prove that $p_M = 1-\Theta(M^{-1})$. We approach this by conditioning on the first step of our random walk. Namely, starting from $S_0 = M$, if $Z_1 \geq 0$, then $Q_t$ necessarily hits $M$ before 0. Therefore, it suffices to analyze the likelihood that $Q_t$ hits $M$ before 0 if $Z_1 < 0$. We formalize this via the law of total probability, which gives us:
\[p_M = 1-\Pr(Z_1 < 0)\Pr(Q_\tau \leq 0 \mid Z_1 < 0, S_0 = M).\]
We provide a closed-form expression for $\Pr(Q_\tau \leq 0 \mid Z_1 < 0, S_0 = M)$ by applying the optional stopping theorem to the martingale $(Q_t, t \geq 1)$, i.e., $\E[Q_\tau \mid Z_1 < 0, S_0 = M] = \E[Q_1 \mid Z_1 < 0, S_0 = M]$.
Conditioning the left-hand side of this equation on the event $\{Q_\tau \leq 0\}$ and solving for $\Pr(Q_\tau \leq 0 \mid Z_1 < 0, S_0 = M)$, we establish that
\[\Pr(Q_\tau \leq 0 \mid Z_1 < 0, S_0 = M) =\frac{\E[Q_\tau \mid S_0 = M, Z_1 < 0, Q_\tau \geq M]-M-\E[Z_1 \mid Z_1 < 0]}{\E[Q_\tau \mid S_0 = M, Z_1 < 0, Q_\tau \geq M]-\E[Q_\tau \mid S_0 = M, Z_1 < 0, Q_\tau \leq 0]},\]
which can be seen to be $\Theta(M^{-1})$ by inspection. Therefore, $p_M = 1-\Theta(M^{-1})$. Symmetric arguments establish that $p_0 = 1-\Theta(M^{-1})$. Plugging these bounds into \eqref{eq:hit_count} in \Cref{Thm:Main_expression_for_inefficiency}, it remains to argue that $E(M)$ and $E(0)$ scale as $\Theta(M)$.

These bounds similarly follow from applications of the optional stopping theorem on appropriately defined martingales. For the upper bound on $E(M)$, we consider the martingale $Q_t^2 - \sigma^2t$, where $\sigma^2 = \E[Z_1^2]$. Then, by the optional stopping theorem,
\begin{align}\label{eq:opt-stop-em-sketch}
\E[Q_\tau^2 \mid S_0 = M]-\sigma^2 E(M) = M^2.
\end{align}
Solving for $E(M)$, and using the fact that $Q_{\tau}^2 \leq (M + Z_{\max})^2$ almost surely, we obtain $E(M) = O(M)$. 

Unfortunately, \eqref{eq:opt-stop-em-sketch} produces too loose a lower bound on $E(M)$. We overcome this difficulty by conditioning on the first step of the process, as we had done before. In particular, suppose $Z_1 \geq 0$. Then, by definition, $\tau = 1$. If $Z_1 < 0$, on the other hand, $\tau = 1 + \tau(M+Z_1)$, where, recall, $\tau(M+Z_1)$ denotes the first time the process hits $0$ or $M$, starting from state $M+Z_1$. Therefore, it suffices to lower bound $\E[\tau(M+Z_{1})\ind{Z_1 < 0}]$. Applying the optional stopping theorem to the martingale $Q_t(M-Q_t) + \sigma^2 t$ yields the following key inequality, for any $S \in [0,M]$:
\[\E[\tau(S) \mid S] \geq \frac{S(M-S)}{\sigma^2}.\]
Plugging $S = M+Z_1 {< M}$ into the above, and integrating over all $Z_1 < 0$, we obtain $\E[\tau(M+Z_1)\ind{Z_1 < 0}] = \Omega(M)$, thereby establishing that $E(M) = \Omega(M)$. The lower bound on $E_0$ follows similar lines; we omit it as such.
\end{proof}

\begin{remark}
An immediate corollary of the Central Limit Theorem is that $\overline{W}$ and $\overline{V}$ scale as $\Omega(1/M^2)$, looser than the lower bound given in \Cref{thm:Inefficient_for_fix_0_mean}. This illustrates the power of \Cref{Thm:Main_expression_for_inefficiency}, and the necessity of a finer-grained analysis of the process's hitting times. 
\end{remark}


\Cref{thm:Inefficient_for_fix_non-0_mean} next establishes that, for any choice of $\alpha \neq \alpha^*$, static allocation policies incur a constant inefficiency, in the long-run. Therefore, the natural proportional allocation policy is optimal amongst all static policies.
\begin{restatable}{theorem}{InefficiencyDrift} 
\label{thm:Inefficient_for_fix_non-0_mean}
    {Any static policy that allocates $\alpha \neq \alpha^*$ in each round incurs}
    $\overline{V} = \Theta(1)$ when
  $\alpha > \frac{\Bmean}{\Nmean}$, and $\overline{W} = \Theta(1)$ when $\alpha < \frac{\Bmean}{\Nmean}$. Therefore, for any $\alpha \neq \alpha^*$, $\Deff = \Theta(1)$.   
\end{restatable}

We provide some intuition behind the result, deferring a formal proof to Appendix \ref{apx:thm-3-proof}. Fix $\alpha \neq \alpha^*$. If $\alpha > \alpha^*$, $S_t$ has negative drift; as a result, one would expect it to spend a constant fraction of time in state 0 in steady state. For $\alpha < \alpha^*$, on the other hand, $S_t$ has positive drift, resulting in a constant fraction of time spent in state $M$. In the former case, the long-run average stockout rate will be large; in the latter, the long-run average overflow dominates.
\subsection{Envy-Inefficiency Trade-off Under a Bang-Bang Policy} 
\label{ssec:Bang-Bang_Policy} 

In this section we demonstrate that simple adaptive policies --- namely, the popular class of \ALG policies from the stochastic control literature \citep{scokaert2002constrained,cerf1995bang} --- can yield an exponential improvement in efficiency, at the price of limited envy. At a high level, our algorithm proceeds as follows: when the store is more than half empty, our algorithm allocates slightly less than $\mu_{\BDist}/\mu_{\NDist}$ in order to avoid the risk of a stockout; when it is more than half full, it allocates slightly more than $\mu_{\BDist}/\mu_{\NDist}$ in order to avoid the risk of overflow.

Formally, our algorithm takes as input a target envy level $\Delta > 0$. Then, for each round $t\in\N$ it sets:
\begin{equation}
\label{eq:bang_bang_pt_2}
A_t = \begin{cases}
\frac{\mu_{\mathcal{B}}}{\mu_{\mathcal{N}}} - \frac{\Delta}{2} &\text{if } S_{t-1} < \frac{M}{2} \\
\frac{\mu_{\mathcal{B}}}{\mu_{\mathcal{N}}} + \frac{\Delta}{2} &\text{if } S_{t-1} \geq \frac{M}{2}.
\end{cases}
\end{equation}

Note that under this algorithm, we have:
\[
\begin{aligned}
\mathbb{E}[Z_t \mid S_{t-1} \ge M/2] 
&= \mu_{\mathcal{B}} - \mu_{\mathcal{N}}\!\left(\frac{\mu_{\mathcal{B}}}{\mu_{\mathcal{N}}} + \frac{\Delta}{2}\right)
= -\,\frac{\mu_{\mathcal{N}} \Delta}{2}, \\[4pt]
\mathbb{E}[Z_t \mid S_{t-1} < M/2] 
&= \mu_{\mathcal{B}} - \mu_{\mathcal{N}}\!\left(\frac{\mu_{\mathcal{B}}}{\mu_{\mathcal{N}}} - \frac{\Delta}{2}\right)
= \;\;\frac{\mu_{\mathcal{N}} \Delta}{2}.
\end{aligned}
\]
Thus, the expected inventory drift is negative when the inventory is above $M/2$, and positive when it is below $M/2$. In other words, \ALG($\Delta$) uses a slack of $\Delta$ in fairness to induce an $M/2$-reverting drift of magnitude $\mu_{\mathcal{N}} \Delta/2$. This steers the process away from the boundary states $\{0,M\}$, reducing the stationary boundary mass (and hence overflow/stockout) relative to static policies for which $\Delta = 0$.
 \Cref{thm:bang-bang_inefficiency_bound} below formalizes this intuition, establishing that such a simple $M/2$-reverting policy yields exponential gains in efficiency.

\begin{restatable}{theorem}{BangBangInefficiency} 
\label{thm:bang-bang_inefficiency_bound}
    Fix $\Delta \in (0, 2\Bmean/\Nmean)$, and suppose there exist strictly positive $\epsilon$ and $\delta$ such that:
    \begin{align}\label{eq:thm-cond}
    \Pr(B_t - N_t(\Bmean / \Nmean + \Delta / 2) \geq \epsilon) \geq \delta \quad \text{ and } \quad \Pr(B_t - N_t(\Bmean / \Nmean - \Delta / 2) \leq - \epsilon) \geq \delta.
    \end{align}
    {Then, the} \ALG($\Delta$) policy achieves
    $$\overline{W} = e^{-\Omega(\Delta M)} \;\;\;\;\;\;\;\; \text{ and }\;\;\;\;\;\;\;\; \overline{V} =e^{-\Omega(\Delta M)}.$$
    Therefore, $\Deff = e^{-\Omega(\Delta M)}$.
\end{restatable}

Before proving our main result, we comment on the conditions on $\Delta$ required for \Cref{thm:bang-bang_inefficiency_bound} to hold. The upper bound $\Delta < 2\Bmean/\Nmean$ simply enforces that $A_t > 0$ for all $t$. \Cref{eq:thm-cond} on the other hand guarantees that the inventory drift $Z_t$ is lower (resp., upper) bounded by a positive (resp., negative) constant with constant probability. These are mild regularity conditions that ensure that the hitting times are well-defined, thereby allowing us to apply \Cref{Thm:Main_expression_for_inefficiency}. They effectively impose an upper bound on $\Delta$ with respect to the distributional primitives $\BDist$ and $\NDist$. We provide a sufficient condition on $\Delta$ for these to hold in Appendix \ref{apx:delta-cond}, noting that the range of $\Delta$ for which this condition holds increases with the variance of $\BDist$.

\Cref{thm:bang-bang_inefficiency_bound} implies the optimality of the popular class of Bang-Bang policies which, in contrast to the optimal policies derived in existing work on fair allocation~\citep{sinclair2021sequential,banerjee2024online}, are not forward-looking. While the optimality of these policies is well-established for {\it linear} stochastic control problems~\citep{scokaert2002constrained,cerf1995bang}, our inefficiency metric and dynamics are non-linear in the allocations; it is therefore not a priori obvious that this simple policy class is scaling-optimal in our setting. We prove the theorem below.

\begin{proof}[Proof]
For notational convenience, we analyze \ALG($2\Delta$) instead of \ALG($\Delta$) throughout. 

Note that our algorithm only ever allocates $A_t \in \{\Bmean / \Nmean - \Delta, \Bmean / \Nmean + \Delta\}$. By our assumptions on $\Delta$, there exist $\epsilon, \delta > 0$ such that:
\begin{align*}
&\Pr\left(B_t-N_t\left(\frac{\mu_{\BDist}}{\mu_{\NDist}} + \Delta \right) \geq \epsilon \right) \geq \delta \implies \Pr\left(B_t-N_t\left(\frac{\mu_{\BDist}}{\mu_{\NDist}} - \Delta \right) \geq \epsilon \right) \geq \delta,\\
\text{and } &\Pr\left(B_t-N_t\left(\frac{\mu_{\BDist}}{\mu_{\NDist}} - \Delta \right) \leq -\epsilon \right) \geq \delta \implies \Pr\left(B_t-N_t\left(\frac{\mu_{\BDist}}{\mu_{\NDist}} + \Delta \right) \leq -\epsilon \right) \geq \delta.
\end{align*}
Hence, the conditions of \Cref{lem:policy_constants} (see Appendix \ref{apx:aux-res}) hold, and $H_M$ and $H_0$ exist. By \Cref{Cor:Main_expression_for_inefficiency}, it suffices to upper bound these two terms. 

Recall, by \Cref{Thm:Main_expression_for_inefficiency}, $H_M \leq \frac{1}{E(M)}$ and $H_0 \leq \frac{1}{E(0)}$. We show that
$E(M) = e^{\Omega(\Delta M)}$, and therefore \mbox{$H_M = e^{-\Omega(\Delta M)}$}. The proof that $E(0) = e^{\Omega(\Delta M)}$ (and therefore $H_0 = e^{-{\Omega}(\Delta M)}$) follows similarly. We omit it as such.

In order to bound $E(M)$, we define the following sequence of stopping times {(see \Cref{fig:flip_stopping_times} for a diagram)}. Let $\tau_0 = 0$. For all $k \geq 1$, let $\tau_k$ denote the $k$-th time $Q_t$ crosses state $M/2$. Formally:
\[\tau_{k} = \inf\left\{t > \tau_{k-1}: \left(Q_t \geq \frac{M}{2} \text{ and } Q_{t-1} < \frac{M}{2}\right) \text{ or } \left(Q_t<\frac{M}{2} \text{ and } Q_{t-1} \geq \frac{M}{2}\right)\right\}.\]
Moreover, let $\kappa = \max\{k \in \mathbb{N}: \tau_k \leq \tau\}$ be the number of times $Q_t$ crosses $M/2$ before hitting one of the boundaries. We use the convention that $\kappa = 0$ if $\tau < \tau_1$. Notice that $\tau \geq \kappa$. To see this, observe that
\begin{align*}
\tau = \sum_{s=1}^{\kappa}\left(\tau_s-\tau_{s-1}\right) + \tau-\tau_{\kappa} \geq \kappa,
\end{align*}
since $\tau_s - \tau_{s-1} \geq 1$ by definition, and moreover $\tau_\kappa \leq \tau$. Therefore, $E(M) := \E[\tau] \geq \E[\kappa]$.

\begin{figure}[!t]
\centering
\scalebox{1}{
\input{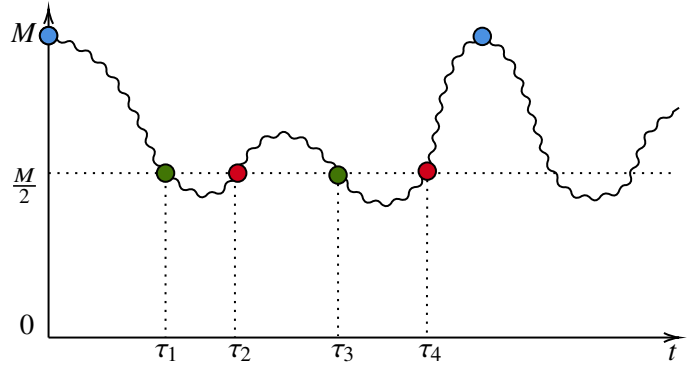}
}
\caption{Sequence of stopping times $(\tau_k)_{k \in [N]}$ used in the proof of \cref{thm:bang-bang_inefficiency_bound}. The process starts at $S_0 = M$.  
For $k$ odd (green marker), the process crosses $\tfrac{M}{2}$ from above; for $k$ even (red marker), it crosses $\tfrac{M}{2}$ from below. Here, $\kappa = 4$.
}
\label{fig:flip_stopping_times}
\end{figure}

To see why $\E[\kappa]$ is large under our Bang-Bang policy, notice that the stopping times $\tau_k$ represent the times at which the Bang-Bang policy switches its allocation. Hence, between any two stopping times, the Bang-Bang policy is a static policy whose inventory level drifts {\it toward} $M/2$. As a result, $Q_t$ will hit $M/2$ before the closest boundary with high probability; it will then flip over to the other side of $M/2$, and the same phenomenon will occur. Putting these facts together, the Bang-Bang policy enforces oscillation around $M/2$, with high probability, therefore resulting in a high value of $\E[\kappa]$.

We formalize this intuition by first lower bounding $\E[\kappa]$ using the tail sum formula, i.e.,:
\begin{align}\label{eq:tail-sum}
\E[\kappa] &= \sum_{k\geq 1}\Pr(\kappa \geq k) \geq \sum_{k \geq 1}\Pr(\kappa\geq k \mid \kappa \geq k-1)\Pr(\kappa \geq k-1) \geq \Pr\left(\kappa \geq 1\right)\sum_{k\geq 1}\prod_{i=2}^{k}\Pr\left(\kappa \geq i \mid \kappa \geq i-1\right).
\end{align}

In order to analyze the above expression, we define $\Ttaulow = \inf\{t > 0: Q_t \not\in (0, \frac{M}{2})\}$ and $\Ttauhigh = \inf\{t > 0: Q_t \not\in [ \frac{M}{2}, M)\}$. The following lemma --- the workhorse of our theorem --- establishes that (i) the likelihood that $Q_t$ crosses $M/2$ before hitting $M$ scales as $\Omega(1/M)$, and (ii) once $Q_t$ is close to $M/2$, it hits the closest boundary before crossing $M/2$ with probability that is exponentially decreasing in $\Delta M$. 

\begin{restatable}{lemma}{HitProbabilityLemma}\label{lem:bound_on_hit_probabilities}
There exists a constant $C > 0$ independent of $M, \Delta$ such that
\begin{align}
\pr\!\left(\left. Q_{\Ttauhigh} < \frac{M}{2}\,\right|\, S_0=M\right)
&\;\ge\; \frac{C}{M}. \label{eq:hit-prob-from-M}
\end{align}
Moreover, there exist positive constants $C_1, C_2$, independent 
of $M, \Delta$, such that
\begin{align}
\pr\!\left(\left. Q_{\Ttaulow} \ge \frac{M}{2}\,\right|\, S_0=S\right)
&\;\ge\; 1 -  e^{-C_1\Delta M},
&&\text{for } S \in \bigl[M/2 - Z_{\min},\, M/2\bigr), \label{eq:hit-prob-below} \\
\pr\!\left(\left. Q_{\Ttauhigh} < \frac{M}{2}\,\right|\, S_0=S\right)
&\;\ge\; 1 - e^{-C_2\Delta M},
&&\text{for } S \in \bigl[M/2,\, M/2 + Z_{\max}\bigr] \label{eq:hit-prob-above}.
\end{align}
\end{restatable}

To derive these two facts, we construct appropriately defined supermartingales, to which we then apply the optional stopping theorem; we defer the formal proof to Appendix \ref{apx:bang-bang-aux}. 

We use \cref{eq:hit-prob-from-M} to bound $\Pr(\kappa \geq 1)$. In particular, by definition, the event $\{\kappa \geq 1\}$ is the event that $Q_t$ crosses $M/2$ before it hits $M$, given $S_0 = M$. Therefore, 
\[\Pr(\kappa \geq 1) = \Pr\left(Q_{\Ttauhigh} < M/2 \mid S_0 = M\right) \geq C/M,\] 
for some constant $C > 0$.

Consider now $k \geq 2$, and suppose first that $k$ is even. Conditional on the event $\{\kappa \geq k-1\}$, $Q_{\tau_{k-1}} < M/2$. Moreover, the event $\{\kappa \geq k\}$ in this case is the event that $Q_t$ crosses $M/2$ before it hits 0. ``Restarting'' the process at $\tau_{k-1}$, we then have:
\begin{align*}
\Pr(\kappa \geq k \mid \kappa \geq k-1, Q_{\tau_{k-1}}) = \Pr\left({Q}_{\Ttaulow} \geq M/2 \mid S_0 = Q_{\tau_{k-1}}\right).
\end{align*}
Since $k$ is even, we have $Q_{\tau_{k-1}-1} \geq M/2$ on the event $\{\kappa \geq k-1\}$. Therefore, $Q_{\tau_{k-1}} \in [M/2-Z_{\min},M/2]$. Applying Lemma~\ref{lem:bound_on_hit_probabilities},  \cref{eq:hit-prob-below} to the above, we obtain:
\begin{align*}
\Pr(\kappa \geq k \mid \kappa \geq k-1, Q_{\tau_{k-1}}) \geq 1-e^{-C_1\Delta M},
\end{align*}
for some $C_1 > 0$. Integrating over all possible $Q_{\tau_{k-1}} \in [M/2-Z_{\min},M/2]$, we have that $\Pr(\kappa \geq k \mid \kappa \geq k-1) \geq 1-e^{-C_1\Delta M}$.

Suppose now that $k$ is odd. In this case, $Q_{\tau_{k-1}} \geq M/2$. Moreover, the event $\{\kappa \geq k\}$ is the event that $Q_t$ crosses $M/2$ before it hits $M$. Similarly here, ``restarting'' the process at $\tau_{k-1}$, we have:
\begin{align*}
\Pr(\kappa \geq k \mid \kappa \geq k-1, Q_{\tau_{k-1}}) = \Pr\left({Q}_{\Ttauhigh} < M/2 \mid S_0 = Q_{\tau_{k-1}}\right) \geq 1-e^{-C_2\Delta M},
\end{align*}
for some $C_2 > 0$, where we have used the fact that $Q_{\tau_{k-1}-1} < M/2 \implies Q_{\tau_{k-1}} \leq M/2+Z_{\max}$ and applied Lemma~\ref{lem:bound_on_hit_probabilities},  \cref{eq:hit-prob-above}. Integrating over all possible $Q_{\tau_{k-1}} \in [M/2,M/2+Z_{\max}]$, we obtain \mbox{$\Pr(\kappa \geq k \mid \kappa \geq k-1) \geq 1-e^{-C_2\Delta M}$}. 

Let $C_3 = \min\{C_1, C_2\}$. Plugging these lower bounds into \eqref{eq:tail-sum}, we have:
\begin{align*}
\E[\kappa] \geq \frac{C}{M}\sum_{k\geq 1}\left(1-e^{-C_3\Delta M}\right)^{k-1} = \frac{C}{M}\cdot e^{C_3\Delta M} \geq e^{C_4\Delta M},
\end{align*}
for some $C_4 > 0$, for large enough $M$. Recalling that $E(M) \geq \E[\kappa]$, the proof is complete.
\end{proof}

\section{Impossibility Result}
\label{sec:impossibility}

We now provide a lower bound on the inefficiency of {\it any} policy that achieves an ex-post envy of $\Delta$. Our results will then imply that the \ALG policy analyzed in \Cref{ssec:Bang-Bang_Policy} lies on the envy-inefficiency Pareto frontier in a small-$\Delta$ regime, and is therefore order-wise optimal.

\begin{restatable}{theorem}{LowerBound} \label{thm:main_lower_bound}
Suppose $B_t \sim \text{Bernoulli}(1/2)$ and $N_t = 1$ almost surely, for all $t \in \mathbb{N}$.  Fix $0 < \Delta \leq 1/9$. Then, for any policy such that $\max\limits_{t,t'}|A_t - A_{t'}| = \Delta$, {$\Deff = e^{-O(\Delta M)}$.}
\end{restatable}

We highlight two key features of \Cref{thm:main_lower_bound}. First, the instance we construct for our lower bound highlights that the Pareto frontier is driven by {\it supply}, rather than demand, uncertainty. Second, the proof of our results considers the wider class of \emph{anticipatory} policies wherein the decision-maker has full knowledge of future randomness. This shows that the impossibility is due to inevitable stochastic fluctuations in the environment, rather than an information-theoretic loss due to the fact that the policy is non-anticipatory. 

\begin{proof}
Fix $a > 0$, and consider any allocation policy such that $A_t \in [a, a+\Delta]$. 
We partition the time horizon into epochs of length $L \in \mathbb{N}$. For $k \in \mathbb{N}$, let $\Bk$ be the total budget arrival in epoch $k$; note that $\Bk\sim\text{Binomial}(L,1/2)$. Moreover, let $\Yplusk \coloneqq \Bk - (a+\Delta)L$ and $\Yminusk  \coloneqq  \Bk - aL$. 

Notice that the policy necessarily incurs an overflow of at least 1 whenever $\Yplusk \geq M+1$. This follows from the fact that the unreflected inventory level $Q_t \geq S_0 + \Yplusk$ at the end of epoch $k$ (since $a+\Delta$ is the {\it most} it can allocate in each period). 
Therefore, on this event $Q_t \geq M+1$. We use this to lower bound $\overline{W}$ as follows:
\begin{align*}
\label{eq:w_v_lower_bounds}
\overline{W} & = \limsup_{T \rightarrow \infty} \frac{1}{T} \sum_{t=1}^T W_t \geq \lim_{T \rightarrow \infty} \frac{1}{T} \sum_{k=1}^{\floor{T / L}} \frac{\Ind{\Yplusk \geq M+1}}{\floor{T/L}} = \Pr(\Yplusk \geq M+1),
\end{align*}
where the final equality follows from the Strong Law of Large Numbers.

In a similar vein, the policy incurs an underage of at least one whenever $\Yminusk \leq -1$. Therefore:
\begin{align*}
\overline{V} & = \limsup_{T \rightarrow \infty} \frac{1}{T} \sum_{t=1}^T V_t \geq \lim_{T \rightarrow \infty} \frac{1}{T} \sum_{k=1}^{\floor{T/L}} \frac{\Ind{\Yminusk \leq -1}}{\floor{T/L}} = \Pr(\Yminusk \leq -1).
\end{align*}

Hence, it suffices to lower bound $\Pr(\Yplusk \geq M+1)$ {or} $\Pr(\Yminusk \leq -1)$ {to obtain a lower bound on $\Deff$}. We partition our analysis into three cases.

\textbf{Case 1:} $M+1+(a+\Delta)L \leq L/2$. Note that we can always choose $L$ to satisfy this inequality as long as $a < \frac12-\Delta$. In this case:
\[
\Pr(\Yplusk \ge M{+}1)
= \Pr(\Bk \geq M+1 + (a + \Delta)L) \geq \Pr(\Bk \geq L/2) \geq \frac12,
\]
since $\Bk \sim \text{Binomial}(L, \tfrac12)$ is symmetric around its mean $L/2$. Therefore, in this case we have $\overline{W} \geq \frac12$.

\medskip
\textbf{Case 2:} $-1 + aL \geq L/2$. Note that we can always choose $L$ to satisfy this inequality as long as $a > \frac12$. In this case, we have:
\[
\Pr(\Yminusk \le -1) = \Pr(\Bk \leq aL - 1) \geq \Pr(\Bk \leq L/2) \geq \frac12,
\]
again, by symmetry of $\Bk$. Therefore, $\overline{V} \geq \frac{1}{2}$. 

\medskip 

Cases 1 and 2 imply that, if $a > \frac12$ or $a < \frac12-\Delta$, we can always choose $L$ such that the policy incurs constant inefficiency. It remains to prove the claim for $a \in \left[\frac12-\Delta,\frac12+\Delta\right]$.

\medskip 

\textbf{Case 3:} $a \in \left[\frac12-\Delta,\frac12+\Delta\right]$. To obtain positive lower bounds on $\Pr(\Yminusk \le -1)$ and $\Pr(\Yplusk \geq M+1)$, we require $aL-1 \geq 0$ and $M+1+(a+\Delta)L \leq L$, which hold for $L \geq \max\left\{\frac1a, \frac{M+1}{1-a-\Delta}\right\}$, which is satisfied for any $L \geq \frac{M+1}{\frac12-2\Delta}$.

We bound $\Pr(\Yminusk \leq -1)$. We have:
\begin{align*}
\Pr(\Yminusk \leq -1) = \Pr(\Bk \leq aL-1) = \Pr(L-\Bk \geq L(1-a)+1) =  \Pr(L-\Bk \geq L/2+L(1/2-a)+1) \\ \geq \Pr(L-\Bk \geq L/2+L\Delta+1),
\end{align*}
where the inequality follows from $a \geq \frac12-\Delta$. By symmetry, $L-\Bk\sim \text{Binomial}(L, 1/2)$.

We use the following bound on the Binomial tail.
\begin{proposition}[Adapted from Proposition 7.3.2 in \citet{alon2016probabilistic}]\label{prop:binomial-tail}
Suppose $L$ is even. Then, for any integer $t \in (0, L/8]$:
\[
\Pr\left(\Bk \,\ge\, \tfrac{L}{2}+t\right) \;\ge\; \tfrac{1}{15}\, e^{-16t^2/L}.
\]
\end{proposition}
Letting $t = L\Delta+1$, it suffices to show that $\frac{(L\Delta+1)^2}{L} = O(M\Delta)$, which holds for $L = cM/\Delta$, for some $c > 0$. (Without loss of generality, we assume $L$ is even.)

We conclude by showing that there exists $c > 0$ such that (i) $L\Delta+1 \leq L/8$, and (ii) $L \geq \frac{M+1}{\frac12-2\Delta}$. For the first property, we have: 
\begin{align*}
L\left(\frac18-\Delta\right) \geq 1 \iff c \geq \frac{\Delta}{M\left(\frac18-\Delta\right)}.
\end{align*}
The above holds for any $c \geq 8$, as long as $\Delta \leq \frac{1}{9}$. For the second property, we require $c$ such that:
\begin{align*}
c \geq \frac{(M+1)\Delta}{M\left(\frac12-2\Delta\right)},
\end{align*}
which holds for any $c \geq 0.8$, for $\Delta \leq \frac19$.  
\end{proof}

\section{Extension to Multiple Resources, Multiple Agent Types}
\label{sec:extensions}

Up to this point, our analysis has considered a single-resource, homogeneous-agent setting. In this section, we outline how our results extend naturally to more complex models involving multiple resources and agent types. While the proofs require technical modifications of the arguments already presented, the core insights remain unchanged.  Since the single-resource setting is a special case of this more general model, the lower bound from \cref{thm:main_lower_bound} still applies.

\paragraph{Generalized Model.} The decision-maker manages a warehouse with total capacity $M > 0$, and $K$ different types of resources. Initially there are $S_{0,k} \in \mathbb{R}_{\geq 0}$ units of each resource $k \in [K]$ (with $\norm{S_0}_1 \leq M$).  At the beginning of each period, the decision-maker receives an additional budget of $B_t \in \mathbb{R}^K_{\geq 0}$ resources, drawn i.i.d. from a distribution $\BDist$ supported on a subset of $[0,B_{\max}]^K$, with mean $\mu_{\BDist} \in \mathbb{R}^K_{\geq 0}$ and strictly positive variance for each resource.  She then observes a random number $N_{t,\theta} \in \mathbb{N}$ new agents of type $\theta$.  Agents of type $\theta$ have utility 
$u(x, \theta) = w_{\theta}^\top x$ for an allocation $x \in \mathbb{R}^K_{\geq 0}$, where $w_{\theta} \in \mathbb{R}^K_{>0}$ corresponds to the preference weights that type $\theta$ assigns to each resource. 
We assume $N_{t,\theta}$ is drawn i.i.d. from a distribution $\mathcal{N}_{\theta}$ supported on a subset of $[0,N_{\max}]$. We use $\Theta$ to denote the set of all possible types, and let $\mu_{\NDist, \theta} \in \mathbb{R}_{\geq 0}$ denote the respective means of $(\mathcal{N}_{\theta}, \theta \in \Theta)$, with  $\Nmean = \sum_{\theta \in \Theta} \mu_{\mathcal{N}, \theta}$. The decision-maker must then decide on an allocation $A_{t, \theta} \in \mathbb{R}^{K}_{\geq 0}$ of the resources to allocate to each of the $N_{t,\theta}$ agents.  If the on-hand inventory exceeds $M$ at the end of the period, the decision-maker decides which of the resources to throw out in order to satisfy the capacity constraint. Let $W_{t,k}$ and $V_{t,k}$ respectively denote the overflow and stockout quantities of resource $k$ at the end of period $t$. As before, these quantities respectively represent the amount of resource $k$ wasted or purchased via an outside option at the end of the period. Formally:
\begin{equation*}
W_{t,k} = (S_{t-1,k} + B_{t,k} - \sum_{\theta} N_{t, \theta} A_{t,\theta,k}  - S_{t,k})^+ \quad,\quad V_{t,k} = (S_{t,k} - (S_{t-1,k} + B_{t,k} - \sum_{\theta} N_{t, \theta} A_{t,\theta,k} ))^+.
\end{equation*}

We have the following metrics of inefficiency and envy in the multi-resource setting.
\begin{definition}[Inefficiency]
\label{def:efficiency_multi}
Given per-unit overflow cost $h \in \mathbb{R}_{> 0}$ and per-unit stockout cost $b \in \mathbb{R}_{> 0}$, the \emph{inefficiency} of a sequence of allocation decisions $(A_{t,\theta}, t\in \N, \theta\in\Theta)$ is given by:\footnote{It is a straightforward extension to have different overflow and stockout costs $h_k$ and $b_k$ per resource type.}
\begin{equation*}
\Deff = \lim \sup_{T \rightarrow \infty} \left( \frac{1}{T} \sum_{t=1}^T \sum_{k=1}^K hW_{t,k} + bV_{t,k} \right) =  h\overline{W} + b\overline{V},
\end{equation*}
where
\begin{equation*}
\overline{W} = \lim \sup_{T \rightarrow \infty} \frac{1}{T} \sum_{t=1}^T \sum_{k=1}^K W_{t,k} \quad,\quad \overline{V} = \lim  \sup_{T \rightarrow \infty} \frac{1}{T} \sum_{t=1}^T \sum_{k=1}^K V_{t,k}.
\end{equation*}
\end{definition}

\smallskip 

\begin{definition}[Ex-post envy]
\label{def:fairness_multi}
The \emph{ex-post envy} of allocation decisions \mbox{$(A_{t,\theta}, t\in \N, \theta\in\Theta)$} is given by:
\begin{equation*}
\Dfair \;=\; \sup_{t,\theta, t', \theta': N_{t, \theta} > 0, N_{t', \theta'} > 0} \, | u(A_{t', \theta'}, \theta) - u(A_{t,\theta}, \theta)|.
\end{equation*}
\end{definition}

\paragraph{The \ALG policy.} We now describe the modification to the \ALG policy for the multi-resource, multi-type setting.  At a high level, we split the capacity $M$ into $K$ ``virtual stores'', each with capacity $M/K$.  We then run \ALG independently on every resource $k \in [K]$ using the same fairness tolerance $\Delta$.  Formally, for fixed $\Delta > 0$, the \ALG($\Delta$) policy allocates:
\begin{equation}
\label{eq:bang_bang_pt_2_multi}
A_{t,\theta,k} = \begin{cases}
\frac{\mu_{\mathcal{B},k}}{\mu_{\mathcal{N}}} - \frac{\Delta}{2} & S_{t-1,k} < \frac{M}{2K} \\
\frac{\mu_{\mathcal{B},k}}{\mu_{\mathcal{N}}} + \frac{\Delta}{2} & S_{t-1,k} \geq \frac{M}{2K}.
\end{cases}
\end{equation}
At the end of the period, the policy chooses any $S_{t,k}$ that minimizes the instantaneous overflow and stockout costs, such that $\norm{S_t}_1 \leq M$. Note that for $K = 1$, we recover our original \ALG($\Delta$) policy. Moreover, it is easy to see that this policy yields an envy bound of $O(\Delta)$ since:
\[
\Dfair = \max_{t,\theta, t’, \theta': N_{t, \theta} > 0, N_{t', \theta'} > 0} | u(A_{t', \theta'}, \theta) - u(A_{t,\theta}, \theta)| \leq \max_{\theta} \sum_{k} w_{\theta,k} \Delta = O(\Delta).
\]

The following theorem establishes that the fundamental envy-inefficiency trade-off from the single-resource setting is unchanged.
\begin{restatable}{theorem}{ExtensionThm}
\label{thm:extensions}
Fix $\Delta \in (0, 2\min_k \frac{\mu_{\BDist,k}}{\Nmean})$, and suppose there exist strictly positive $\epsilon$ and $\delta$ such that, for all $k \in [K]$:
\[
    \Pr\left(B_{t,k} - \left(\sum_\theta N_{t,\theta}\right)(\mu_{\mathcal{B},k} / \Nmean + \Delta / 2) \geq \epsilon\right) \geq \delta \quad \text{ and } \quad \Pr\left(B_{t,k} - \left(\sum_\theta N_{t,\theta}\right)(\mu_{\mathcal{B},k} / \Nmean - \Delta / 2) \leq - \epsilon\right) \geq \delta.
\]
Then the \ALG($\Delta$) policy {achieves}
    $$\overline{W} = e^{-\Omega(\Delta M)} \;\;\;\;\;\;\;\; \text{ and }\;\;\;\;\;\;\;\; \overline{V} = e^{-\Omega(\Delta M)}.$$
    As a result, $\Deff = e^{-\Omega(\Delta M)}$.  Moreover, the static policy that allocates $\frac{\mu_{\BDist,k}}{\mu_{\NDist}}$ of resource $k$ to each type $\theta$ individual in each period {achieves}
    {$$\overline{W} = \Theta\left( \frac{1}{M} \right) \;\;\;\;\;\;\;\; \text{ and }\;\;\;\;\;\;\;\; \overline{V} = \Theta\left(\frac{1}{M}\right).$$
    As a result, $\Deff = \Theta(1 / M)$.} 
\end{restatable}
We prove the result in Appendix \ref{apx:extensions-proof}. In Appendix \ref{apx:extensions-improvement} we discuss how to incorporate type-dependent preferences in the allocation decisions.

\section{Numerical Experiments}
\label{sec:simulations}

In this section we perform extensive synthetic experiments to illustrate the robustness of our policies, particularly with respect to our modeling assumptions --- namely, that $(N_t, B_t)$ are bounded and i.i.d. over time.
The main questions we seek to answer through our experiments are:
\begin{itemize}[nosep,leftmargin=0.5cm]
    \item {(I, \cref{sec:experiment_m_scale,sec:experiments_delta_scale})} {\em Impact of Capacity and Fairness Constraints.} How does increasing $M$ and $\Delta$ impact the performance of the \ALG($\Delta)$ policy?
    \item {(II, \cref{sec:experiments_time_varying})} {\em Distribution Variability.}  How well do our policies perform across different distributions? In particular, what is the impact of the tail behavior of the demand and supply distributions on the envy-inefficiency trade-off?
    Do our insights hold when $(N_t, B_t)$ is time-varying?
\end{itemize}
We defer experiments treating the extension to multiple resources to \cref{app:experiments_multi_resources}.
Across all simulations we compare the static policy which allocates $\Bmean/\Nmean$ in each period (referred to as \Static) to the \ALG($\Delta$) policy, for various values of $\Delta$.  {Note that when $\Dfair=0$, \ALG($\Dfair$) reduces to \Static.}\footnote{See~\url{https://github.com/seanrsinclair/Inventory-Based-Shared-Resource-Allocation} for the code.} 

\paragraph{Experimental Setup.}  Across all experiments, we let $h = b = 1$ (results for other values of $(h,b)$ are similar and omitted as such).  {We test 20 values of the capacity $M$, equally spaced in $[10, 100]$, and 20 values of $\Dfair$, equally spaced in $[0, 0.5]$}.  We moreover consider the following supply and demand distributions:
\begin{itemize}
    \item {\bf Truncated Normal}: $N_t \sim \Norm(\Bmean, \sigma_b^2)^+$ and $B_t \sim \Norm(\Nmean, \sigma_n^2)^+$ as truncated normal random variables, with $\Bmean = \Nmean = 5$, $\sigma_b = \sigma_n = 1$.
    \item {\bf Poisson and Exponential}: $N_t \sim \Poisson(\Nmean)$, $B_t \sim \Exponential(1 / \Bmean)$ for $\Nmean = \Bmean = 5$.
    \item {\bf Time-Varying Normal}: $N_t \sim \Norm(\Bmean(t), \sigma_b^2)^+$ and $B_t \sim \Norm(\Nmean(t), \sigma_n^2)^+$ as truncated normal random variables.  {We assume $\Bmean(t)$ and $\Nmean(t)$ are periodic, with cycle length $C$ (i.e., $\Bmean(t) = \Bmean(t+C)$ for all $t$; similarly for $\Nmean(t)$).}
\end{itemize}
Note that our simulated supply and demand distributions have unbounded support (contrary to the boundedness assumption in our basic model); this allows us to test the robustness of our results to differences in tail behavior.

\paragraph{Algorithm Evaluation.} We run $100$ replications for each experiment.  To evaluate the algorithm's inefficiency $\Deff$, we assume a time horizon of $T = 10,000$ and average results over all replications. We plot $\log(\Deff)$ for all experiments, capping $\Deff$ at $10^{-4}$ to avoid numerical precision issues. 
See~\cref{sec:experiment_details} for further details. 

\subsection{Impact of $M$}
\label{sec:experiment_m_scale}

\begin{figure}[t]
\centering
\begin{subfigure}[t]{0.48\linewidth}
    \centering
    \includegraphics[width=\textwidth]{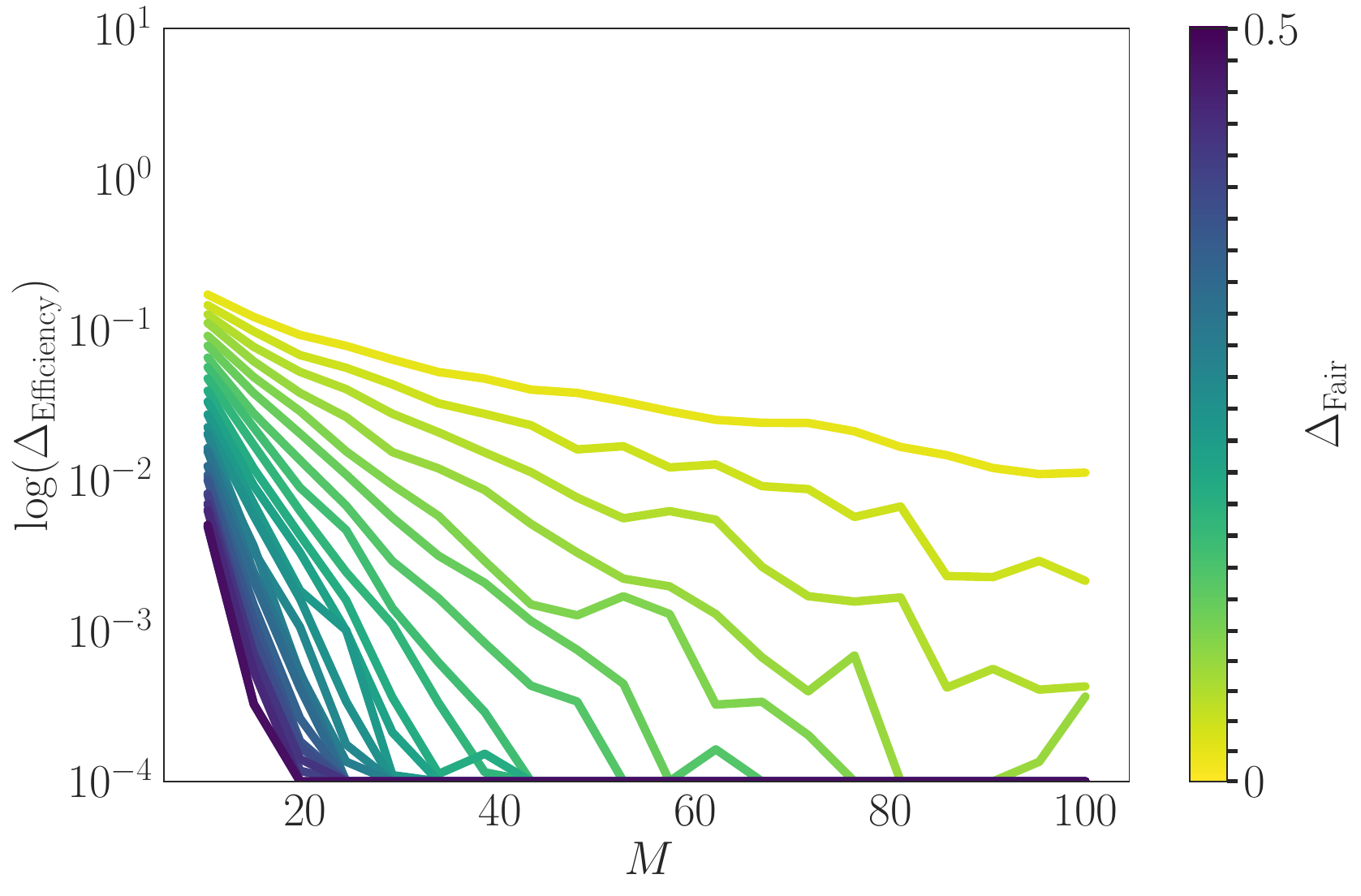}
    \caption{\small Truncated Normal demand and donations}
    \label{fig:normal_M}
\end{subfigure}
\hfill
\begin{subfigure}[t]{0.48\linewidth}
    \centering
    \includegraphics[width=\textwidth]{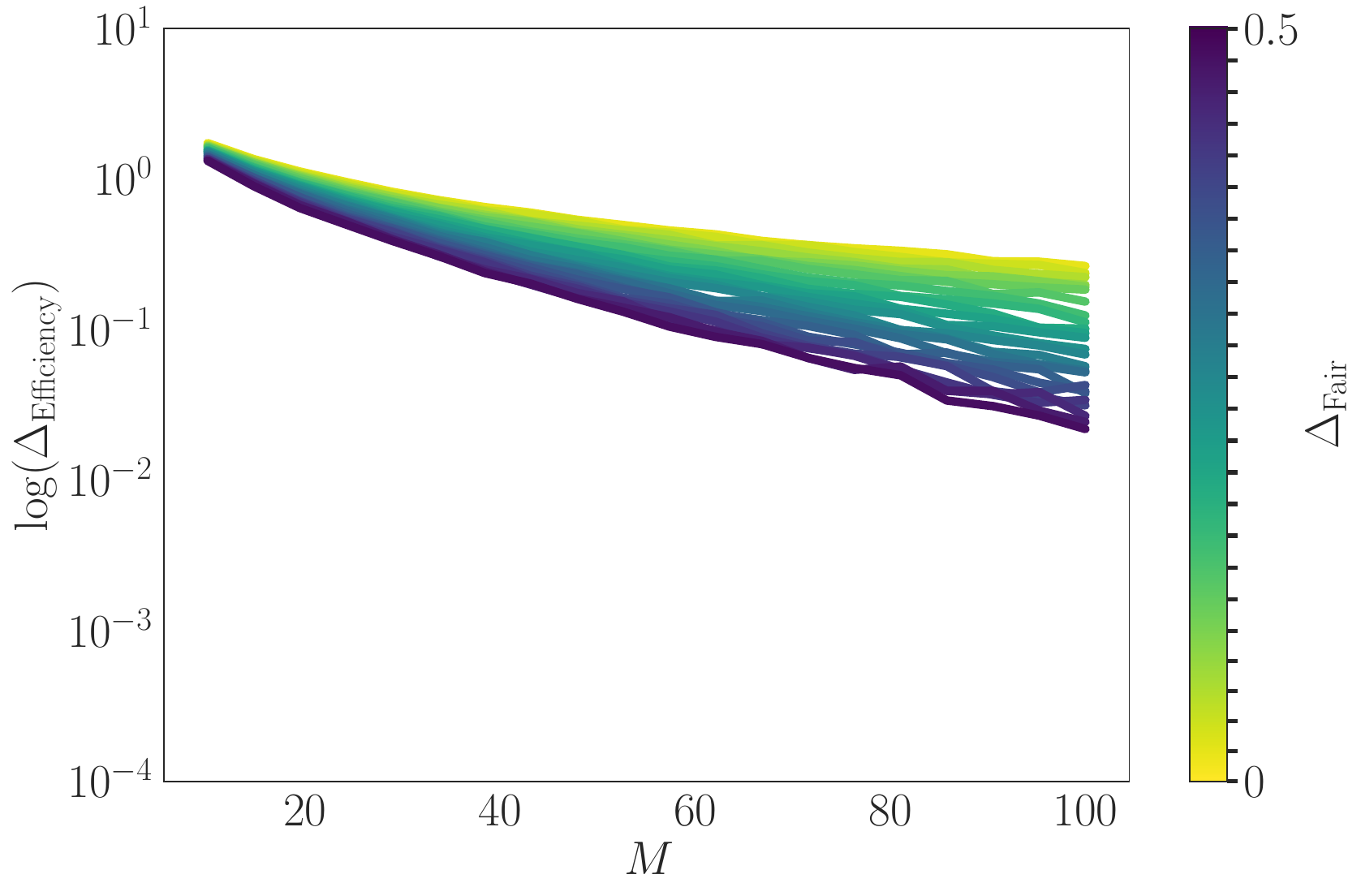}
    \caption{\small Poisson demand and Exponential donations}
    \label{fig:poisson_M}
\end{subfigure}
\caption{{$\log(\Deff)$ vs. $M$ under \ALG($\Dfair$), for different values of $\Dfair \in [0,0.5]$.} Under sub-Gaussian demand/supply, the phase transition close to $\Dfair=0$ is clearly visible; the trade-off is less dramatic under heavier-tailed demand/supply distributions in~\Cref{fig:poisson_M}.}
\label{fig:change_M}
\end{figure}

We first investigate the performance of our benchmark policies as we scale the capacity $M$. 

\paragraph{Summary.} \Cref{fig:change_M} illustrates the gains in efficiency as we scale $M$. For small values of $M$, inefficiency remains high across all values of $\Dfair$ due to frequent stockouts and overflows from the insufficient capacity.  As $M$ increases, we observe exponential improvement in efficiency. \Cref{fig:poisson_M} shows the exponential scaling is present, but much less pronounced under heavier tails.  Moreover, \(\Deff\) is markedly worse under heavier-tailed distributions, reflecting more frequent boundary hits (stockouts and overflows).

These plots highlight an important practical insight --- increasing storage capacity has {\em sharply diminishing returns} on efficiency unless fairness constraints are also adjusted. In low-capacity systems, strict fairness (i.e. $\Dfair = 0$) leads to severe inefficiency, whereas allowing a small amount of flexibility ($\Dfair > 0$) dramatically improves performance (as illustrated in \Cref{fig:normal_M}).  We emphasize that \ALG leverages this relaxation in fairness by adapting allocations based on inventory levels, achieving near-optimal efficiency while simultaneously maintaining bounded fairness.  In contrast, the static policy ensures fairness but incurs significant inefficiency.

\subsection{Impact of $\Dfair$}
\label{sec:experiments_delta_scale}

\begin{figure}[t]
\centering
\begin{subfigure}[t]{0.48\linewidth}
    \centering
    \includegraphics[width=\textwidth]{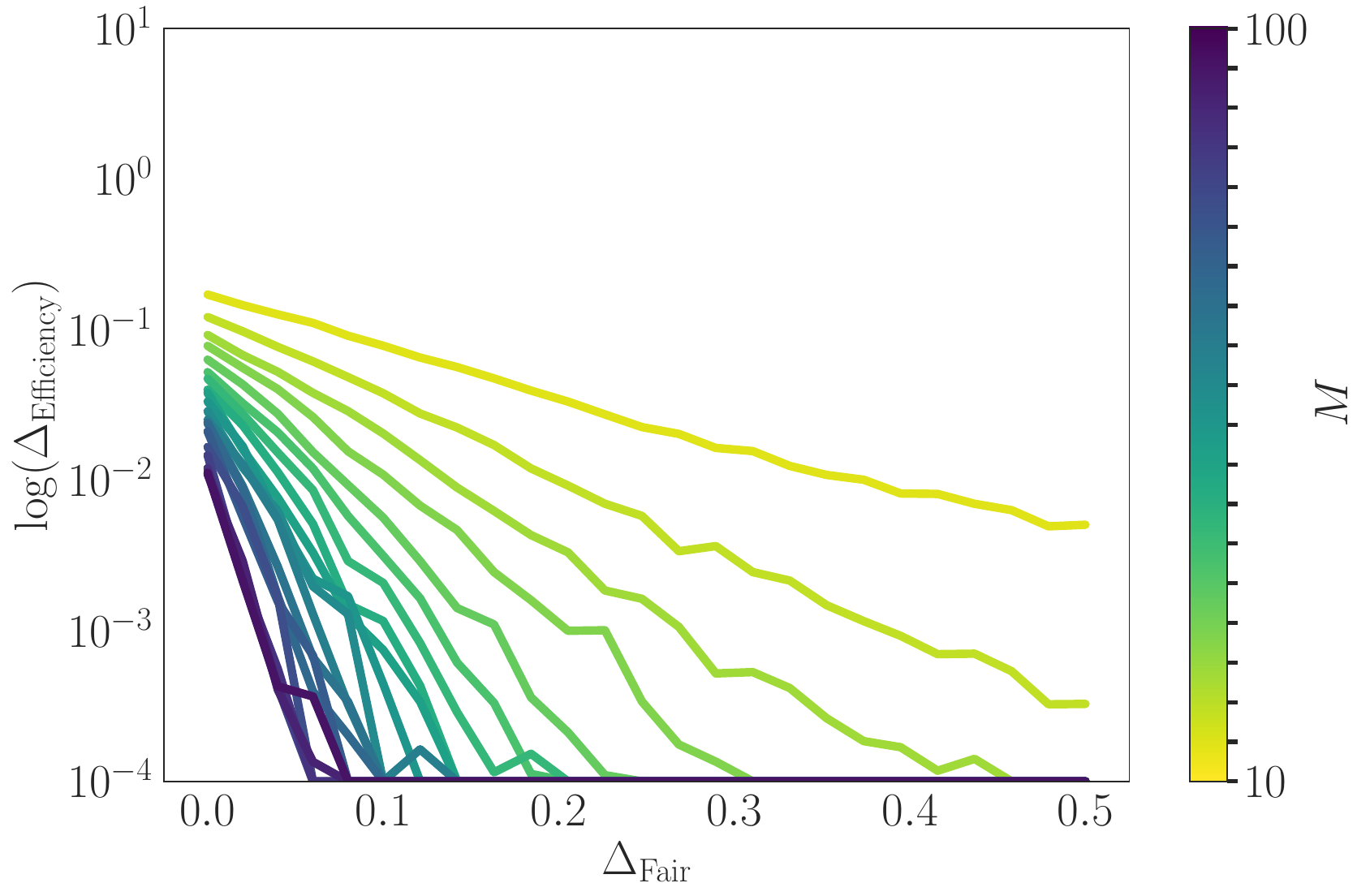}
    \caption{\small Truncated Normal demand and donations}
    \label{fig:normal_delta}
\end{subfigure}
\hfill
\begin{subfigure}[t]{0.48\linewidth}
    \centering
    \includegraphics[width=\textwidth]{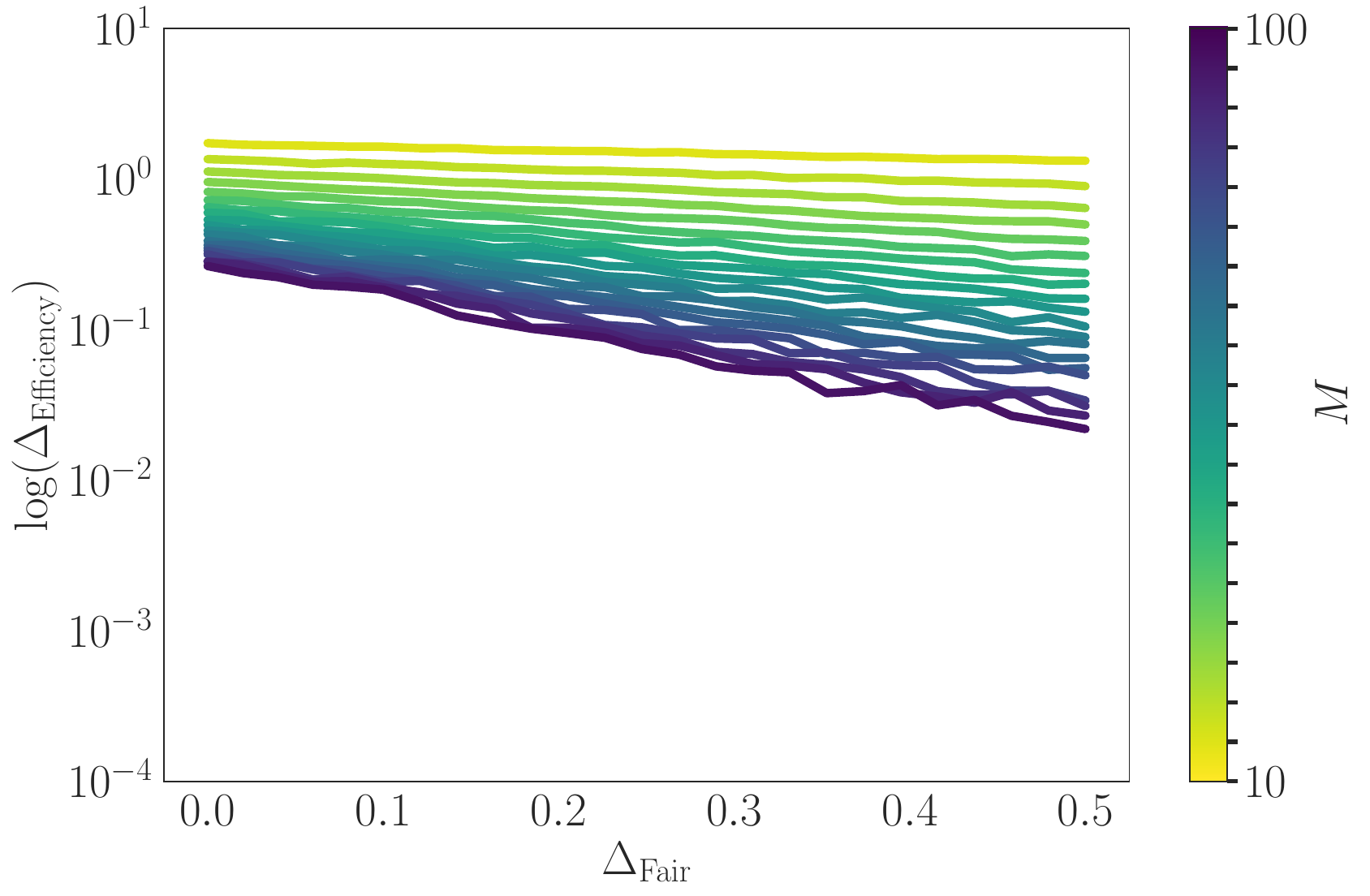}
    \caption{\small Poisson demand and Exponential donations}
    \label{fig:poisson_delta}
\end{subfigure}
\caption{{$\log(\Deff)$ vs. $\Dfair$ under $\ALG(\Dfair)$, for different values of $M \in [10,100]$.}  Under sub-Gaussian demand/supply, the phase transition close to $\Dfair=0$ is clearly visible; the exponential scaling is still present with heavier-tailed distributions, but without the sharp transition close to $\Dfair=0$.}
\label{fig:change_Delta}
\end{figure}

We next study the performance of the \ALG policy as a function of $\Dfair$, for various values of $\Dfair$.  We again evaluate the performance on truncated Normal demand and supply, as well as subexponential demand/supply distributions.

\paragraph{Summary.} \Cref{fig:change_Delta} illustrates the performance improvement for $\Deff$ under larger values of $\Dfair$ while holding the capacity constraint $M$ fixed.  The results again highlight a clear phase transition: when $\Dfair = 0$, inefficiency remains consistently high across all values of $M$, indicating the substantial cost of strict fairness.  However, even a slight relaxation of fairness ($\Dfair > 0$) yields significant efficiency improvements. Similar to our results in \cref{sec:experiment_m_scale}, we see that efficiency gains are nonlinear, with the largest improvements occurring  for small increases in $\Dfair$.  This suggests that a modest relaxation of fairness constraints can lead to disproportionately large efficiency gains.  {However, as before, these gains depend on the tails of the supply and demand distributions, and in the case of subexponential tails in~\Cref{fig:poisson_delta}, the gains in efficiency are more modest.}

\subsection{Impact of Time-Varying Distributions}
\label{sec:experiments_time_varying}

\begin{figure}[t]
\centering
\begin{subfigure}[t]{0.48\linewidth}
    \centering
    \includegraphics[width=\textwidth]{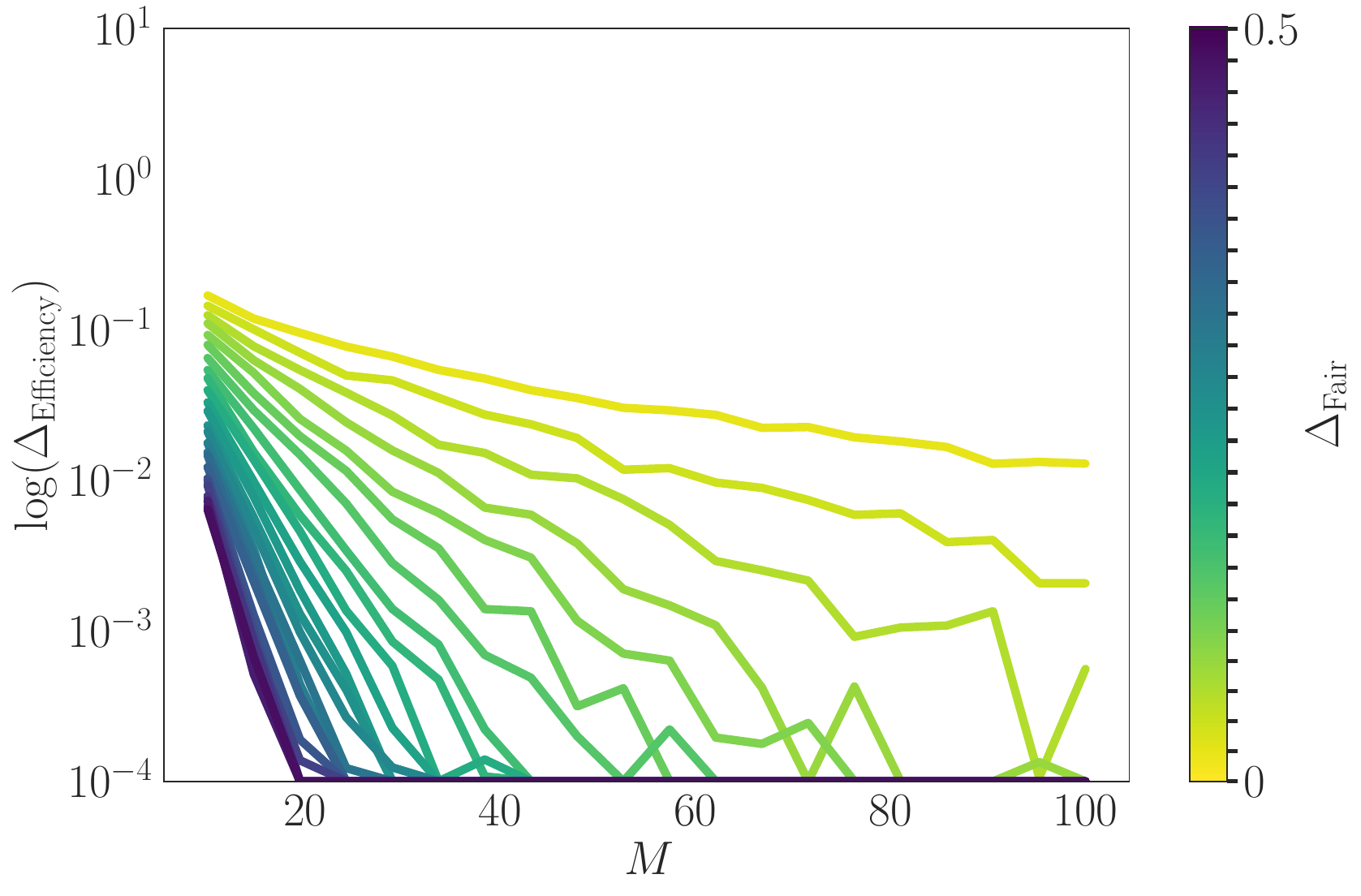}
    \caption{\small $\log(\Deff)$ vs $M$}
    \label{fig:time_varying_M}
\end{subfigure}
\hfill
\begin{subfigure}[t]{0.48\linewidth}
    \centering
    \includegraphics[width=\textwidth]{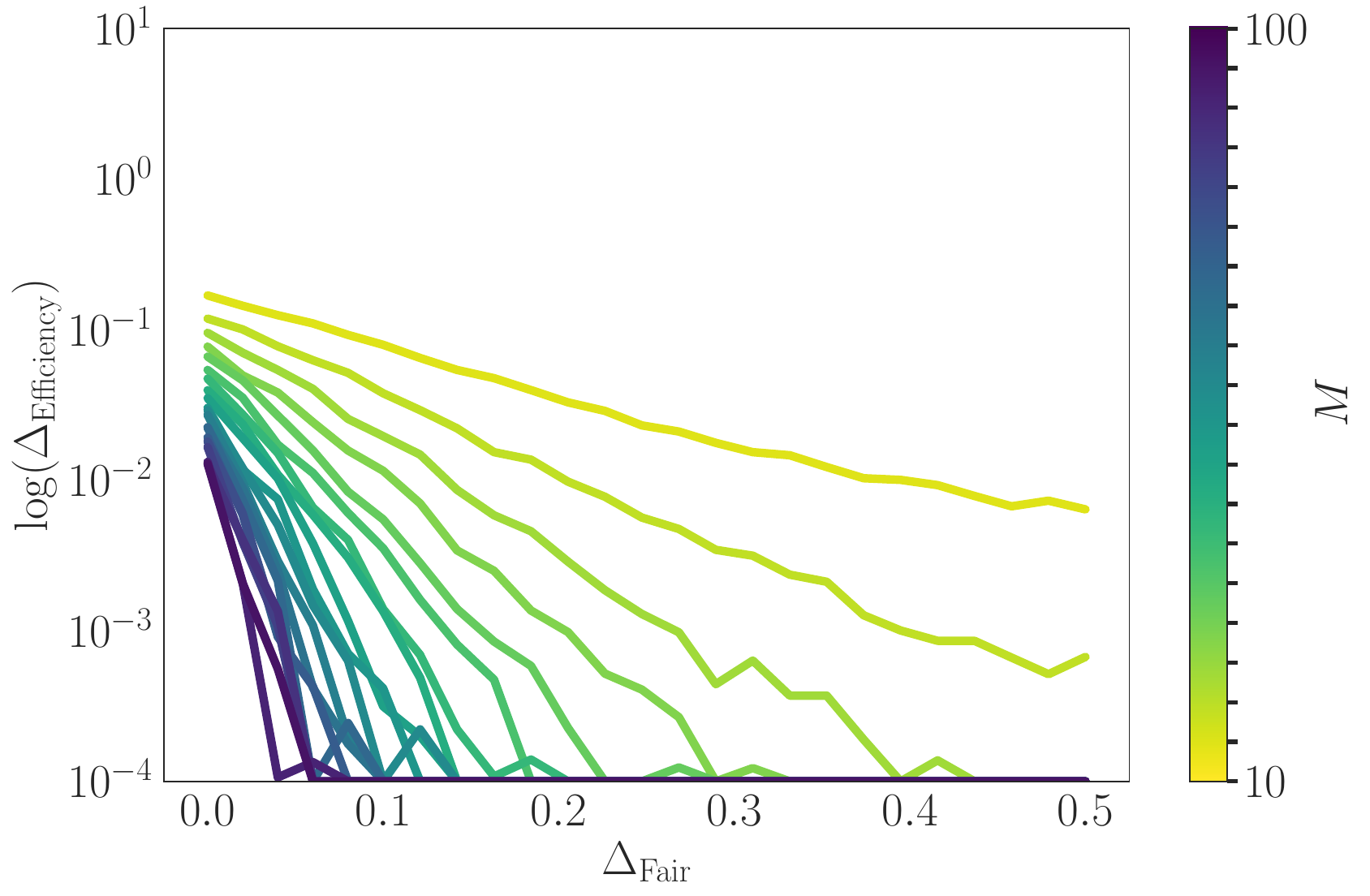}
    \caption{\small $\log(\Deff)$ vs $\Dfair$}
    \label{fig:time_vary_delta}
\end{subfigure}
\caption{{Impact of nonstationarity on inefficiency.} The exponential scaling and sharp phase transition near $\Dfair=0$ are still clearly visible under periodic supply and demand distributions.}
\label{fig:time_varying}
\end{figure}

While our theoretical results focused on settings where supply and demand are i.i.d. across time, real-world resource allocation problems often involve non-stationary demand/supply.
In this section, we test the robustness of our policy to nonstationary demand and supply processes.  Specifically, we evaluate how our proposed policies perform when donations and demand are periodic.
To account for this periodicity, we modify \ALG as follows:
\begin{equation}
    \label{eq:alg_modification}
    A_t = \begin{cases}
        \frac{\sum_{t \in C} \Bmean(t)}{\sum_{t \in C} \Nmean(t)} + \frac{\Dfair}{2} & S_{t-1} \geq \frac{M}{2} \\
        \frac{\sum_{t \in C} \Bmean(t)}{\sum_{t \in C} \Nmean(t)} - \frac{\Dfair}{2} & S_{t-1} < \frac{M}{2}.
    \end{cases}
\end{equation}
This is identical to the \ALG policy where the reference level $\Bmean / \Nmean$ is modified to be the cumulative supply divided by cumulative demand over a cycle.
We again evaluate the performance on values of $\Dfair \in [0,0.5]$ and $M \in [10, 100]$.

\paragraph{Summary.} \Cref{fig:time_varying} demonstrates that the key envy-inefficiency trade-off observed in the i.i.d. setting persists, even when donations and demand are cyclical.
Specifically, we find that the phase transition in inefficiency ($\Deff$) as a function of fairness ($\Dfair$) remains intact. As in the i.i.d. case, enforcing strict fairness ($\Dfair = 0$) leads to high inefficiency, whereas allowing a small relaxation ($\Dfair > 0$) significantly improves performance. Additionally, the benefits of increased capacity $M$ remain evident.

\section{Conclusion}
\label{sec:conclusion}

In this work, we studied fundamental trade-offs between fairness and efficiency in repeated resource allocation problems under stochastic replenishments with capacity constraints.  Our results demonstrate that while strict fairness ($\Dfair = 0$) leads to significant inefficiency, even a small relaxation of fairness ($\Dfair > 0$) can dramatically improve performance, with inefficiency decreasing exponentially as both ex-post envy and capacity increase.  We provided both theoretical guarantees and empirical validation, showing that our proposed allocation policy, \ALG, effectively balances fairness and efficiency.  

Our work suggests several open avenues, including extending the insights to non-i.i.d. donation and demand models, and exploring fairness notions beyond envy-freeness. From a technical perspective, while our results establish exponentially decaying inefficiency of our \ALG policy, they do not guarantee the best constant in the exponent. Establishing this (at least in the large-deviations sense), or finding better policies under other natural scalings, are important and promising open questions.

\medskip

{\noindent\textbf{Acknowledgments.} The authors would like to thank Ziv Scully for helpful comments on a preliminary version of our manuscript.}

\newpage

\bibliographystyle{informs2014} 
\bibliography{references} 


%
%
%

\newpage

\AtBeginEnvironment{APPENDICES}{%
  \renewcommand{\theHsection}{appendix.\Alph{section}}%
  \renewcommand{\theHsubsection}{\theHsection.\arabic{subsection}}%
  \renewcommand{\theHsubsubsection}{\theHsubsection.\arabic{subsubsection}}%
  \renewcommand{\theHfigure}{\theHsection.\arabic{figure}}%
  \renewcommand{\theHtable}{\theHsection.\arabic{table}}%
  \renewcommand{\theHequation}{\theHsection.\arabic{equation}}%
  \renewcommand{\theHtheorem}{\theHsection.\arabic{theorem}}%
}
\makeatother

\crefalias{section}{appendix}
\begin{APPENDICES}
\OneAndAHalfSpacedXI 
\section{Omitted Derivations from \Cref{ssec:warmup}}\label{mc:deriv}

\begin{proposition}
\label{prop:mc_stationary_measure}
Let $0 < \Delta < 2p$ with $p \leq 1/2$ and set $\rho = \frac{2p+\Delta}{2p-\Delta}$.  Then, the stationary distribution of the Markov chain in \Cref{fig:biased_gamblers_ruin_chain} is given by:
\[
\pi[0] = \pi[M] = \left(2 \cdot \frac{\rho^{M/2} - 1}{\rho - 1} + \rho^{M/2}\right)^{-1} \quad \pi[i] = \rho^{i}\pi[0] \, \text{ for } i \leq M/2 \quad \pi[i] = \rho^{M-i}\pi[0] \, \text{ for } i > M/2.
\]
\end{proposition}
\begin{proof}
The steady-state equations for this birth-death chain are given by:
\begin{align*}
\pi[i] = 
\begin{cases}
 \pi[i]\left(1-\left(p+\frac\Delta2\right)\right) + \pi[i+1]\left(p-\frac\Delta2\right) \quad & i = 0\\
\pi[i-1]\left(p+\frac\Delta2\right) + \pi[i](1-2p) + \pi[i+1]\left(p-\frac\Delta2\right) \quad & i \in \{1,\ldots,M/2-1\}\\
 \pi[i-1]\left(p+\frac\Delta2\right) + \pi[i](1-2p+\Delta) + \pi[i+1]\left(p+\frac\Delta2\right) \quad & i = M/2\\
\pi[i-1]\left(p-\frac\Delta2\right) + \pi[i](1-2p) + \pi[i+1]\left(p+\frac\Delta2\right)  \quad &i \in \{M/2+1,\ldots,M-1\} \\
\pi[i-1]\left(p-\frac\Delta2\right) + \pi[i]\left(1-\left(p+\frac\Delta2\right)\right) \quad & i = M.
\end{cases}
\end{align*}
Let $\rho = \frac{2p+\Delta}{2p-\Delta}$. Solving for $\pi[i]$, we obtain:
\begin{align*}
\pi[i] = \begin{cases}
\rho^i \pi[0] & i \leq M/2 \\
\rho^{M-i}\pi[0] & i > M/2.
\end{cases}
\end{align*}
Plugging in these expressions into the normalizing constraint $\sum_{i=0}^{M}\pi[M] = 1$, we obtain:
\begin{align*}
\pi[0]\left(\sum_{i=0}^{M/2}\rho^i + \sum_{i=M/2+1}^{M}\rho^{M-i}\right) = 1 &\iff \pi[0]\left( 2 \sum_{i=0}^{M/2-1}\rho^i + \rho^{M/2}\right) = 1 \\
&\iff \pi[0]\left(2 \cdot \frac{\rho^{M/2}-1}{\rho-1}+\rho^{M/2}\right) = 1.
\end{align*}
Solving for $\pi[0]$ and noting that $\pi[0] = \pi[M]$, we obtain the result.
\end{proof}

\section{Omitted Proofs from \Cref{ssec:timehomogen}}\label{apx:4.1-proof}

\subsection{Auxiliary Results}\label{apx:aux-res}

\MainCorr*

\begin{proof}
We first argue that for any $t \in \mathbb{N}$, we have:
\begin{align}\label{eq:as-bounds}
(Z_t(A_t))^+ \cdot \ind{S_{t-1} = M} \leq W_t \leq   Z_{\max} \cdot \ind{S_t = M} \quad \text{and} \quad
(Z_t(A_t))^-\mathds{1}\{S_{t-1} = 0\} \leq V_t \leq Z_{\min}\mathds{1}\{S_t = 0\}.
\end{align}
To see this, note that in every period $t$ such that $S_{t-1} + B_{t} - N_t A_t \geq M$, we have $S_t = M$, and incur overflow
\[W_t = S_{t-1} + B_{t} - N_t A_t - M \leq B_t-N_t A_t \leq Z_{\max},\] where the final inequality follows from the fact that $S_{t-1} \in [0,M]$.
Since $W_t = 0$ in any round $t$ such that $S_t < M$, we obtain $W_t \leq Z_{\max} \cdot \ind{S_t = M}.$ 

Consider now any round $t$ such that $S_{t-1} = M$. Then, in this round \mbox{$W_t = (B_{t} - N_t A_t)^+ = (Z_t(A_t))^+$}. Hence, 
\begin{align}\label{eq:wt-as-lb}
W_t \geq (Z_t(A_t))^+ \cdot \ind{S_{t-1} = M}.
\end{align}
Similar arguments give us the bounds on $V_t$.

\smallskip 

We now use \eqref{eq:as-bounds} to bound $\overline{W}$ and $\overline{V}$. The fact that $\overline{W} \leq Z_{\max} H_M$ follows immediately from the almost sure upper bound on $W_t$. Similarly for the fact that $\overline{V} \leq Z_{\min} H_0$.

For the lower bounds, define the random variable $X_t = \left(Z_t(\alpha_M)^+ - \Exp{Z_t(\alpha_M)^+}\right) \Ind{S_{t-1} = M}.$ 
The sequence $(X_t, t \in \mathbb{N})$ satisfies the following properties:
\begin{enumerate}
\item $\mathbb{E}[X_t \mid X_1,\ldots, X_{t-1}] = 0$: This follows from the fact that $Z_t(\alpha_M)$ is independent of $S_{t'}$, for all $t' < t$.
\item $\lim_{T\to\infty}\sum_{t=1}^Tt^{-2}\mathbb{E}[X_t^2] < +\infty$: This follows from the fact that $\mathbb{E}[X_t^2]\leq {(Z_{\max}+Z_{\min})^2}$. 
\end{enumerate}
Applying the Strong Law of Large Numbers for Martingales (Theorem 1 in \citet{csorgHo1968strong}) to $X_t$, we have:
\begin{align*}
\lim_{T\to\infty}\frac1T\sum_{t=1}^TX_t = 0, \quad \text{almost surely}.
\end{align*}

Applying this to $\overline{W}$ and \Cref{eq:as-bounds}, we have:
\begin{align*}
    \overline{W} = \limsup_{T \rightarrow \infty} \frac{1}{T} \sum_{t=1}^T W_t & \geq \limsup_{T \rightarrow \infty} \frac{1}{T} \sum_{t=1}^T Z_t(\alpha_M)^+ \Ind{S_{t-1} = M} \\
    & \geq \limsup_{T \rightarrow \infty} \frac{1}{T} \sum_{t=1}^T X_t + \liminf_{T \rightarrow \infty} \frac{1}{T} \sum_{t=1}^T \Exp{Z_t(\alpha_M)^+} \Ind{S_{t-1} = M} \\
    & {= \Exp{Z_1(\alpha_M)^+} \liminf_{T \rightarrow \infty} \frac{1}{T} \sum_{t=1}^T \Ind{S_{t} = M} = \Exp{Z_1(\alpha_M)^+} H_M,}
\end{align*}
where the second equality uses the fact that $(B_t, N_t)$ are i.i.d. across $t$.
\end{proof}

\RandomWalkEquivalence*
\begin{proof}
Recall, $\tau(S)$ is used to denote the first time $S_t \in \{0,M\}$.  We first claim that $Q_t = S_t$ for all $t < \tau(S)$. We prove this by induction. Indeed, $Q_0 = S_0$ by definition. Now, fix $t' < \tau(S)-1$, and suppose $Q_{t'} = S_{t'}$. Then:
\[
Q_{t'+1} = Q_{t'} + Z_{t'+1} =   S_{t'} + Z_{t'+1} = \left.(S_{t'} + Z_{t'+1})\right|_0^M = S_{t'+1},
\]
where the second equality follows from the inductive hypothesis, and the third equality follows from the fact that $t'+1 < \tau(S)$, which implies that $S_{t'+1} \in (0,M)$.

We now argue that $\tau(S) \geq \tau$. For all $t < \tau$, $Q_t = S_t$. Therefore, for all $t < \tau$, $S_t \in (0,M)$, which implies that $\tau(S) \geq \tau$. 

Putting these two inequalities together, we obtain the fact that $\tau = \tau(S)$, and the result follows.
\end{proof}

The following lemma provides conditions on the budget and demand distributions $B_t$ and $N_t$ to ensure that the stopping times defined in \cref{ssec:timehomogen} are well-defined and finite.

\begin{restatable}{lemma}{PolicyBounds}
    \label{lem:policy_constants}
    Fix an allocation policy $\mathcal{A}$.  Suppose there exist strictly positive constants $\epsilon$ and $\delta$ such that, for any allocation $\alpha$ made by the policy at time $t$:
    \[
    \Pr(Z_t(\alpha) \geq \epsilon) \geq \delta \quad \text{ and } \quad \Pr(Z_t(\alpha) \leq - \epsilon) \geq \delta.
    \]
    Then, the following properties hold:
    \begin{enumerate}[label=(\roman*)]
    \item  $E(S) < \infty$ for any $S \in [0,M]$,
    \item $\tau_M(S) < \infty$ and $\tau_0(S) < \infty$, almost surely, for any $S \in [0,M]$,
    \item $p_0 \in (0,1)$ and $p_M (0,1)$, and
    \item $\Exp{Z_t(\alpha)^+} \geq \epsilon \delta$ and $\Exp{Z_t(\alpha)^-} \geq \epsilon \delta$.
    \end{enumerate}
\end{restatable}
\begin{proof}
\noindent\textbf{(i) Finite expected hitting time.} Recall, $\tau_M(S)$ and $\tau_0(S)$ respectively denote the first time the process reaches $M$ or $0$, starting from state $S$. To show that $E(S) < \infty$ it suffices to show that $E_0(S) \coloneqq \Exp{\tau_0(S)} < \infty$ for any $S \in [0,M]$, since $\tau(S) \leq \min\{\tau_M(S), \tau_0(S)\}$. In the remainder of the proof, we omit dependence of all quantities on the initial state $S$, for notational simplicity.

By assumption, there exist $\epsilon$ and $\delta$ such that $\Pr(Z_t(\alpha) \leq -\epsilon) \geq \delta$.  Let $L=\lceil M/\epsilon\rceil$.
By the tail-sum formula,
\begin{align*}
\E[\tau_0(S)] = \sum_{k = 0}^\infty \sum_{l=0}^{L-1} \Pr\big(\tau_0(S)> kL + l\big) \leq \sum_{k = 0}^\infty \sum_{l=0}^{L-1} \Pr\big(\tau_0(S)> kL\big) &= L\sum_{k=0}^{\infty}\mathbb{P}\left(\tau_0(S) > kL\right).
\end{align*}
By the law of total probability, and the fact that $\tau_0(S) \leq (k-1)L \implies \tau_0(S) \leq kL$, we have:
\begin{align}\label{eq:expect-tau0}
\mathbb{P}\left(\tau_0(S) \leq kL\right) &= \mathbb{P}\left(\tau_0(S) \leq (k-1)L\right) + \mathbb{P}\left(\tau_0(S) \leq kL , \tau_0(S) > (k-1)L\right)
\end{align}

Consider the event $
E^-_L \;:=\; \{Z_1\le -\epsilon,\; Z_2\le -\epsilon,\ldots, Z_L\le -\epsilon\}.$ Conditioning the second summand on $E_L^-$, we have:
\begin{align*}
\mathbb{P}\left(\tau_0(S) \leq kL\right) &\geq \mathbb{P}\left(\tau_0(S) \leq (k-1)L\right) + \mathbb{P}\left(\tau_0(S) \leq kL , \tau_0(S) > (k-1)L \mid E_L^-\right)\mathbb{P}(E_L^-)
\end{align*}

Note that, if $\tau_0(S) > (k-1)L$, by definition $S_t$ has not reached 0 by time $(k-1)L$. However, on the event $E_L^-$, the unreflected process $Q_t$ decreases by at least $L\epsilon \geq M$. Therefore, it must be the case that $\tau_0(S) \leq kL$, independent of the starting state at $(k-1)L$. Moreover, since $(B_t,N_t)$ are independent across time,
\(\Pr(E^-_L)\ge \delta^L\). Putting these two facts together, we obtain:
\begin{align*}
\mathbb{P}\left(\tau_0(S) \leq kL\right) &\geq \mathbb{P}\left(\tau_0(S) \leq (k-1)L\right) + \mathbb{P}\left( \tau_0(S) > (k-1)L \mid E_L^-\right)\delta^L\\
&= 1-(1-\delta^L)\mathbb{P}\left(\tau_0(S)>(k-1)L\right)
\\
\implies \mathbb{P}\left(\tau_0(S) > kL\right) &\leq (1-\delta^L)\mathbb{P}\left(\tau_0(S) > (k-1)L\right) \leq (1-\delta^L)^k,
\end{align*}
where the final inequality follows from the fact that $\mathbb{P}(\tau_0(S) > 0) = 1$.

Plugging this back into \eqref{eq:expect-tau0}, we obtain:
\begin{align*}
\mathbb{E}[\tau_0(S)] \leq L\sum_{k=0}^{\infty}(1-\delta^L)^k = \frac{L}{\delta^L} < \infty.
\end{align*}

An analogous argument with $E^+_L := \{Z_1\ge \epsilon,\ldots,Z_L\ge \epsilon\}$ gives a finite bound for the expected time to hit $M$, i.e. $\Exp{\tau_M(S)} < \infty$.

\medskip
\noindent\textbf{(ii) Finite hitting times almost surely.}
The previous argument established that $\Exp{\tau_M(S)} < \infty$ and $\Exp{\tau_0(S)} < \infty$. This immediately implies that $\tau_M(S)$ and $\tau_0(S)$ are finite almost surely.

\medskip
\noindent\textbf{(iii) Boundary hit probabilities.} We show $p_M\in(0,1)$; the proof for $p_0$ is symmetric.
By definition,
\[
p_M \;=\; \Pr(S_{\tau(M)}=M \,\big|\, S_0 = M).
\]
A one-step return to $M$ occurs whenever $Z_1\ge \epsilon$, since $M+Z_1 > M \implies S_1 := (M+Z_1)\vert_0^M=M$. Hence, $\tau(M) = 1$, with $S_{\tau(M)} = M$.
Therefore, $p_M \ge \Pr(Z_1\ge \epsilon\mid S_0=M)\ge \delta>0$.

To see that $p_M<1$, note that under event $E^-_L$, the process hits 0 in the first $L$ periods (since the cumulative drift is at most $-L\epsilon \leq -M$), before being able to return to $M$. Moreover, this occurs with probability at least
$\Pr(E^-_L\mid S_0=M)\ge \delta^L$.
Therefore $\Pr(S_{\tau(M)}=0\mid S_0=M)\ge \delta^L>0$, which implies $p_M\le 1-\delta^L<1$.
A symmetric argument starting from $S_0=0$ gives $p_0\in(0,1)$.

\medskip
\noindent\textbf{(iv) Positivity of $\Exp{Z_t(\alpha)^+}$ and $\Exp{Z_t(\alpha)^-}$.}
By assumption:
\[
\mathbb{E}\big[ Z(\alpha)^+ \big] 
\ \ge\ \epsilon \cdot \Pr(Z(\alpha) \geq \epsilon) 
\ \ge\ \epsilon \delta.
\]
An analogous argument gives $\mathbb{E}\big[ Z(\alpha)^- \big] \ge \epsilon\delta$.  
\end{proof}

\subsection{Proof of \Cref{Thm:Main_expression_for_inefficiency}}\label{apx:h-thm}

\InefficiencyExpression*
\begin{proof}
We prove the result for $H_M$; the argument for $H_0$ is analogous. Note that $\tau_M(S_0) < \infty$ almost surely; hence we can assume without loss of generality that $S_0 = M$, since this initial time will have no impact on $\overline{W}$.

We decompose the horizon into cycles in which the inventory level begins at $M$, hit zero, and return to $M$.
In more detail, we define a sequence of stopping times $0 = \tau_0,\tau_1,\tau_2,\ldots$ such that the $n$-th renewal cycle of the process begins in round $\tau_{n-1}$ and ends in round $\tau_n-1$, with
\[\tau_n = \inf\{t > \tau_{n-1}: S_t = M \text{ and there exists } \tau_{n-1} < t' < t \text{ such that } S_{t'} = 0\}.\]
Note that $S_t$ may hit 0 or $M$ multiple times within a cycle. One such cycle is depicted in \Cref{fig:renewal}.

Let $L_n = \tau_{n} - \tau_{n-1}$ be the length of the $n$-th renewal cycle and set $R_n = \sum_{t=\tau_{n-1}}^{\tau_n-1} \ind{S_t = M}$ be the number of times in which $S_t = M$ in the cycle. Since $\policy$ is time-homogeneous, 
$(L_n, R_n)$ form a renewal-reward pair when $\mathbb{E}[L_n] < \infty$.\footnote{We prove $\mathbb{E}[L_n] < \infty$ by providing a closed-form expression for $\Exp{L_n}$ below.}
By the Renewal-Reward Theorem, we then have
\begin{equation}\label{eq:renewal-reward_expression_appx}
    \lim_{T \rightarrow \infty}\frac{\sum_{t=1}^T\mathds{1}\{S_t = M\}}{T} \stackrel{a.s.}{=} \frac{\E[R_1]}{\E[L_1]}.
\end{equation} 
By \Cref{eq:renewal-reward_expression_appx}, it suffices to compute $\E[R_1]$ and $\E[L_1]$. 

Consider first the random variable $R_1$. Note that $R_1$ corresponds to the number of times $S_t$ hits $M$ before hitting 0 (including the initial state $S_0 = M$); this follows from the fact that, if it hits $M$ after 0, then this would correspond to the beginning of a new cycle, and therefore not be counted in $R_1$. To compute $\mathbb{E}[R_1]$, consider the counting process for which periods $t > 0$ when $S_t \in \{0,M\}$ are trials; each trial for which $S_t = M$ is a failure, and the first trial for which $S_t = 0$ is a success. Moreover, since $p_{M}$ is the probability that, starting from $S_{0}=M$, the next visit to $\{0,M\}$ is again at $M$, the success probability for this process is exactly $1-p_M$. Then, $R_1-1$ corresponds to the number of failures before success, and is geometrically distributed with mean $\E[R_1-1] = \frac{1-(1-p_M)}{1-p_M} = \frac{p_M}{1-p_M} \implies \E[R_1] = \frac{1}{1-p_M}$.

We now provide a closed-form expression for $\E[L_1]$. Let $\kappa = \inf\{t < \tau_1: S_t = 0\}$ be the first time $S_t = 0$ in the renewal cycle. 
By definition, $\E[L_1] = \E[\tau_1] = \E[\kappa] + \E[\tau_1 - \kappa ]$. We first deduce an expression for $\E[\kappa]$.

To do so, we consider all periods before $\kappa$ for which $S_t = M$. Formally, let $t_1 = 0$; for $n \geq 2$, define $t_n = \inf\{t > t_{n-1}: S_t \in \{0,M\}\}$. Notice that, by definition, $t_{R_1}$ is the last time the inventory level hits $M$ before first hitting 0, and $\kappa = t_{R_1+1}$. Decomposing $\kappa$ into the lengths of these successive boundary hits, we have:
$$\E[\kappa] = \E\left[\sum_{n=2}^{R_1+1}t_n - t_{n-1}\right].$$
Note that, for all $n \in \{2,\ldots,R_1+1\}$, $t_n - t_{n-1}$ are i.i.d. random variables distributed as $\tau(M)$. Moreover, using the interpretation described above, $R_1+1$ corresponds to the number of trials required for first success and has finite mean (assuming $p_M < 1$), {we apply Wald's equation~\citep[Proposition 1.8.1]{resnick2013adventures}:}
\[\E[\kappa] = \E[R_1]\cdot E(M) = \frac{1}{1-p_M} E(M).\]

The expression for $\E[\tau_1-\kappa]$ proceeds similarly. In particular, let $R_0$ denote the number of times $S_t$ hits 0 before hitting $M$ after $\kappa$. Consider again the sequence $(t_n, n\in\mathbb{N})$. By definition, $\tau_1 = t_{(R_1+1)+(R_0+1)} = t_{R_1+R_0+2}$. We then have: 
\[\E[\tau_1-\kappa] = \sum_{n=R_1+2}^{R_0+R_1+2}t_n-t_{n-1}.\]
In this case, for all $n \in \{R_1+2, \ldots ,R_0+R_1+2\}$, $t_n - t_{n-1}$ are i.i.d. random variables distributed as $\tau(0)$. Moreover, by the same arguments as above, $R_0$ is geometrically distributed with mean $\E[R_0] = \frac{p_0}{1-p_0}$. Applying Wald's equation to the above, we obtain:
\[\E[\tau_1-\kappa] = \E[R_0+1]\cdot E(0) = \frac{1}{1-p_0} E(0).\]
Therefore:
\begin{align*}
&\E[L_1] = \frac{1}{1-p_M} E(M) +\frac{1}{1-p_0} E(0)\\ \implies &\frac{\E[R_1]}{\E[L_1]} = \frac{\frac{1}{1-p_M}}{\frac{1}{1-p_M} E(M) +\frac{1}{1-p_0} E(0)} = \frac{1}{E(M) + \frac{1-p_M}{1-p_0}E(0)}.
\end{align*}

The derivation of $H_0$ is entirely analogous and omitted as such.
\end{proof}
\section{Omitted Proofs from \Cref{ssec:fixed_policy}}\label{apx:static}

\subsection{Proof of \Cref{thm:Inefficient_for_fix_0_mean}}\label{apx:thm-2-proof}

The primary technical tool we will apply to show our result will be the optional stopping theorem.

\InefficiencyZeroMean*

\begin{proof}
For ease of notation, we let $Z_t = Z_t(\alpha^*)$, omitting the dependence on $\alpha^*$ throughout. 
By \Cref{Cor:Main_expression_for_inefficiency}, it suffices to bound $H_M$, $H_0$, $\mathbb{E}[Z_1^+]$, and $\mathbb{E}[Z_1^-]$, if $H_M$ and $H_0$ exist. \Cref{Thm:Main_expression_for_inefficiency} guarantees existence of $H_M$ and $H_0$ as long as $E(M) < \infty$, $E(0) < \infty$, $p_M, p_0 \in (0,1)$, and $\tau_M(S_0)$ and $\tau_0(S_0)$ are finite almost surely. By \Cref{lem:policy_constants}, these properties hold as long as there exist $\epsilon > 0, \delta > 0$ such that $\Pr(Z_t \geq \epsilon) \geq \delta$ and $\Pr(Z_t \leq - \epsilon) \geq \delta$. To see why this would be the case for the static policy defined by $\alpha^*$, note that $\Exp{Z_t} = 0$, and $\Var{Z_t} = \Bstd^2 + (\alpha^*)^2 \Nstd^2 > 0$, since $(B_t, N_t)$ are independent and $\Bstd > 0$ by assumption. Therefore, $\Pr(Z_t > 0) > 0$ and $\Pr(Z_t < 0) > 0$. This also immediately establishes that $\E[Z_1^+]$ and $\E[Z_1^-]$ are lower bounded by positive constants. Hence, it suffices to bound $H_M$ and $H_0$ to obtain bounds on $\overline{W}$ and $\overline{V}$, respectively.

In the remainder of the proof we provide upper and lower bounds for $H_M$. Arguments for $H_0$ are entirely analogous; we omit them as such. Recall, by \Cref{Thm:Main_expression_for_inefficiency},
\[H_M = \frac{1}{E(M) + E(0)\cdot \frac{1-p_{M}}{1-p_{0}}}.\]
We bound each of these four quantities separately.

\medskip 
{\bf Bounding $p_M$ and $p_0$.} We first establish that $p_M$ and $p_0$ scale as $1-\Theta(M^{-1})$. 

By \Cref{lem:Random_walk_equivalence_1}, $p_M = \mathbb{P}(Q_\tau \geq M \mid S_0 = M)$. Then, by the law of total probability:
\[
p_M = \Pr(Q_\tau \geq M \mid Z_1 < 0, S_0 = M) \Pr(Z_1 < 0 \mid S_0 = M) + \Pr(Q_\tau \geq M \mid Z_1 \geq 0, S_0 = M) \Pr(Z_1 \geq 0 \mid S_0 = M).
\]
Note that, if $Z_1 \geq 0$, $Q_1 \geq M$; therefore $\tau = 1$ and $\Pr(Q_\tau \geq M \mid Z_1 \geq 0, S_0 = M) = 1$. Hence:
\begin{align}\label{eq:pM}
p_M &= \Pr(Q_\tau \geq M \mid Z_1 < 0, S_0 = M) \Pr(Z_1 < 0 \mid S_0 = M) + 1 - \Pr(Z_1 < 0 \mid S_0 = M) \notag \\
&= 1-\Pr(Z_1 < 0)\left(1-\Pr(Q_\tau \geq M \mid Z_1 < 0, S_0 = M)\right).
\end{align}

Given $Z_1 < 0$, it is easy to see that $Q_1,\ldots,Q_t$ form a martingale with bounded increments. \Cref{lem:opt-stopping-static-1} applies the optional stopping theorem to obtain a closed-form expression for $\Pr(Q_\tau \geq M \mid Z_1 < 0, S_0 = M)$. We defer its proof to Appendix \ref{appsec:martingales}.

\begin{restatable}{lemma}{OptStoppingStatic}\label{lem:opt-stopping-static-1}
    \[\Pr\left(Q_\tau \geq M \mid Z_1 < 0, S_0 = M\right) = \frac{M+\E[Z_1 \mid Z_1 < 0]-\E[Q_\tau \mid S_0 = M, Z_1 < 0, Q_\tau < M]}{\E[Q_\tau \mid S_0 = M, Z_1 < 0, Q_\tau \geq M]-\E[Q_\tau \mid S_0 = M, Z_1 < 0, Q_\tau < M]}.\]
\end{restatable}

Applying \Cref{lem:opt-stopping-static-1} to \eqref{eq:pM}, we have:
\begin{align*}
1-\Pr\left(Q_\tau \geq M \mid Z_1 < 0, S_0 = M\right) = \frac{\E[Q_\tau \mid S_0 = M, Z_1 < 0, Q_\tau \geq M]-M-\E[Z_1 \mid Z_1 < 0]}{\E[Q_\tau \mid S_0 = M, Z_1 < 0, Q_\tau \geq M]-\E[Q_\tau \mid S_0 = M, Z_1 < 0, Q_\tau < M]},
\end{align*}
which implies:
\begin{align*}
p_M &= 1-\Pr(Z_1<0)\cdot \frac{\E[Q_\tau \mid S_0 = M, Z_1 < 0, Q_\tau \geq M]-M-\E[Z_1 \mid Z_1 < 0]}{\E[Q_\tau \mid S_0 = M, Z_1 < 0, Q_\tau \geq M]-\E[Q_\tau \mid S_0 = M, Z_1 < 0, Q_\tau < M]}\\
&= 1- \frac{\left(\E[Q_\tau \mid S_0 = M, Z_1 < 0, Q_\tau \geq M]-M\right)\Pr(Z_1<0)+\E[Z_1^-]}{\E[Q_\tau \mid S_0 = M, Z_1 < 0, Q_\tau \geq M]-\E[Q_\tau \mid S_0 = M, Z_1 < 0, Q_\tau < M]},
\end{align*}
where the second equality follows from the fact that $\E[Z_1^-] = -\E[Z_1\ind{Z_1 < 0}]=-\Pr(Z_1 < 0)\E[Z_1 \mid Z_1 < 0]$.


For the upper bound on $p_M$, use the fact that
$\E[Q_\tau \mid S_0 = M, Z_1 < 0, Q_\tau \geq M] \geq M$
in the numerator, and $\E[Q_\tau \mid S_0 = M, Z_1 < 0, Q_\tau \geq M] \leq M+Z_{\max}$ in the denominator. Therefore:
\begin{align*}
p_M \leq 1-\frac{\E[Z_1^-]}{M+Z_{\max}} = 1-\Omega(M^{-1}).
\end{align*}

The lower bound on $p_M$ follows from the fact that $\E[Q_\tau \mid S_0 = M, Z_1 < 0, Q_\tau \geq M] \leq M+Z_{\max}$ and $\E[Z_1^-] \leq {Z_{\min}}$ in the numerator. In the denominator, we have $\E[Q_\tau \mid S_0 = M, Z_1 < 0, Q_\tau \geq M] \geq M$, and $\E[Q_\tau \mid S_0 = M, Z_1 < 0, Q_\tau < M] \leq 0$. Hence:
\begin{align*}
p_M \geq 1-\frac{{\max\{Z_{\min}, Z_{\max}\}}\left(1+\Pr(Z_1 < 0)\right)}{M} = 1-O(M^{-1}),
\end{align*}
where here we use the fact that $\Pr(Z_1 < 0) < 1$.

Putting these two bounds together, we obtain $p_M = 1-\Theta(M^{-1})$. The proof that $p_0 = 1-\Theta(M^{-1})$ is entirely symmetric; we omit it as such.

Using the above bounds on $p_M$ and $p_0$, we have:
\begin{align*}
H_M = \frac{1}{E(M) + E(0)\cdot \frac{\Theta(1/M)}{\Theta(1/M)}} = \Theta\left(\left(E(M)+E(0)\right)^{-1}\right).
\end{align*}
Hence, it suffices to show that $E(M) + E(0) = \Theta(M)$. We argue this for $E(M)$; the argument for $E(0)$ is analogous.

\medskip

\textbf{Upper bounding $E(M)$.} Consider the random variable $Q_t^2-\sigma^2t$, where $\sigma^2 = \E[Z_1^2]$. We have the following lemma, whose proof we defer to Appendix \ref{appsec:martingales}.

\begin{restatable}{lemma}{MeanZeroMartingale} \label{lem:Mean_0_case_marginales}
    $Q_t^2-\sigma^2 t$ is a martingale with bounded increments on the event $\{\tau > t\}$.
\end{restatable}

Hence, by the optional stopping theorem:
\begin{align*}
\E[Q_\tau^2 \mid S_0 = M] - \sigma^2\E[\tau] = \E[Q_0^2 \mid S_0 = M] = M^2.
\end{align*}
Using the fact that $\E[\tau] = E(M)$ by \cref{lem:Random_walk_equivalence_1} and solving for $E(M)$, we obtain:
\begin{align}\label{eq:eM-static}
\sigma^2 E(M) = \E[Q_\tau^2 \mid S_0 = M]-M^2\implies E(M) = \frac{\E[Q_\tau^2 \mid S_0 = M]-M^2}{\sigma^2}.
\end{align}
We have:
{
\begin{align*}
\E[Q_\tau^2 \mid S_0 = M] &\leq \max\{(M+Z_{\max})^2, Z_{\min}^2\} = (M+Z_{\max})^2,
\end{align*}
for large enough $M$. Here, the first inequality follows from the fact that, by definition of $\tau$, $Q_\tau = Q_{\tau-1} + Z_t \leq M+Z_{\max}$. Similarly, $Q_{\tau} \geq -Z_{\min}$.
} Plugging this into \eqref{eq:eM-static}, we obtain $E(M) = O(M)$.

\medskip 

\textbf{Lower bounding $E(M)$.} Recall, $\tau(S)$ corresponds to the first time $Q_t \not\in (0,M)$, given initial state $S$. Partitioning based on the event $Z_1 \geq 0$, we have:
\begin{align}\label{eq:em-static-2}
E(M) = \Pr(Z_1 \geq 0) + \E[(1+\tau(M+Z_1)) \ind{Z_1 < 0}] = 1+\E[\tau(M+Z_1) \ind{Z_1 < 0}].
\end{align}
The following lemma will allow us to bound $\E[\tau(M+Z_1) \ind{Z_1 < 0}]$. We defer its proof to the end of the section.
\begin{restatable}{lemma}{OneStepStopping}\label{eq:one-step-stopping-time}
For any $S \in [0,M]$, $\E[\tau(S)] \geq \frac{S(M-S)}{\sigma^2}$.
\end{restatable}
Hence, for fixed $Z_1 < 0$, $
\E[\tau(M+Z_1) \mid Z_1] \geq -\frac{(M+Z_1)Z_1}{\sigma^2}.$
Integrating over all $Z_1 < 0$, we obtain:
\begin{align*}
\E[\tau(M+Z_1) \ind{Z_1 < 0}] \geq \frac{1}{\sigma^2}\left(M\E[Z_1^-]-\E\left[\left(Z_1^-\right)^2\right]\right) = \Omega(M)
\end{align*}
Plugging this back into \eqref{eq:em-static-2}, we obtain $E_M = \Omega(M)$, thereby completing the proof.
\end{proof}

\subsubsection{Auxiliary Lemmas}\label{appsec:martingales}

\OptStoppingStatic*
\begin{proof}
By the optional stopping theorem:
\begin{align}\label{eq:opt1}
\E\left[Q_\tau \mid S_0 = M, Z_1 < 0\right] &= \E[Q_1 \mid Z_1 < 0] = \Exp{S_0 + Z_1 \mid Z_1 < 0} = M + \Exp{Z_1 \mid Z_1 < 0}.
\end{align}
We have:
\begin{align*}
&\E\left[Q_\tau \mid S_0 = M, Z_1 < 0\right]\\ &= \E[Q_\tau \mid S_0 = M, Z_1 < 0, Q_\tau \geq M]\Pr\left(Q_\tau \geq M \mid Z_1 < 0, S_0 = M\right)\\
&\quad +\E[Q_\tau \mid S_0 = M, Z_1 < 0, Q_\tau < M]\left(1-\Pr\left(Q_\tau \geq M \mid Z_1 < 0, S_0 = M\right)
\right) \\
&= \E[Q_\tau \mid S_0 = M, Z_1 < 0, Q_\tau < M]\\&\quad + \Pr\left(Q_\tau \geq M \mid Z_1 < 0, S_0 = M\right) \left(\E[Q_\tau \mid S_0 = M, Z_1 < 0, Q_\tau \geq M]-\E[Q_\tau \mid S_0 = M, Z_1 < 0, Q_\tau < M]\right)
\end{align*}
Plugging the above into \eqref{eq:opt1} and solving for $\Pr\left(Q_\tau \geq M \mid Z_1 < 0, S_0 = M\right)$, we obtain:
\begin{align*}
&\Pr\left(Q_\tau \geq M \mid Z_1 < 0, S_0 = M\right) = \frac{M+\E[Z_1 \mid Z_1 < 0]-\E[Q_\tau \mid S_0 = M, Z_1 < 0, Q_\tau < M]}{\E[Q_\tau \mid S_0 = M, Z_1 < 0, Q_\tau \geq M]-\E[Q_\tau \mid S_0 = M, Z_1 < 0, Q_\tau < M]}.
\end{align*}
\end{proof}

\MeanZeroMartingale*
\begin{proof}
   Let $Y_{t} = Q_{t}^2 - \sigma^2 t = \left(\sum_{s=1}^t Z_s + S_0\right)^2 - \sigma^2 t$, and denote by $\mathcal{F}_t$ the natural filtration up to time $t$.
   We have: 
    \begin{align*}
        \E[Y_{t+1} - Y_{t} \mid \mathcal{F}_{t}] & = \E\left[\left(\sum_{s=1}^{t+1} Z_s + S_0\right)^2 - \left(\sum_{s=1}^t Z_s + S_0\right)^2 - \sigma^2 \mid \mathcal{F}_{t}\right]\\
        & = \E\left[\left(\sum_{s=1}^{t+1} Z_s- \sum_{s=1}^t Z_s\right)\left( 2Q_{t} + Z_{t+1} \right) \mid \mathcal{F}_{t}\right] - \sigma^2\\
        & = \E[Z_{t+1}\left(2Q_{t} + Z_{t+1} \right) \mid \mathcal{F}_{t}] - \sigma^2\\
        & = \E[Z^2_{t+1}]- \sigma^2 = 0,
    \end{align*}
    where we have used the fact that $\E[Z_{t+1}Q_t \mid \mathcal{F}_t] = \E[Z_{t+1}]\E[Q_t \mid \mathcal{F}_t] = 0$. Thus, $Y_t$ is a martingale. 
    
    We now show that $Y_t$ has bounded increments with respect to $\tau = \inf\{t > 0 \mid Q_t \not\in (0,M)\}$. Namely,
    \begin{align*}
        |Y_{t+1}-Y_{t}| = |Z_{t+1}\left(2Q_t+Z_{t+1}\right)-\sigma^2| \leq 2MZ_{\max} + \left(\max\{Z_{\max},Z_{\min}\}\right)^2 + \sigma^2 < \infty.
    \end{align*}
\end{proof}

\OneStepStopping*
\begin{proof}
Let $\Phi(x):=x(M-x)$. We claim that $Y_t = \Phi(Q_t) + \sigma^2 t$ is a martingale. Letting $\mathcal{F}_t$ denote the natural filtration as before, we have:
\begin{align*}
\E[Y_{t+1} \mid \F_t]
&=\E\big[ (Q_{t}+Z_{t+1})(M-(Q_{t}+Z_{t+1})) \,\big|\, \F_t \big] + \sigma^2(t+1) \\
&=\Phi(Q_t) + \E\big[(M-2Q_t)Z_{t+1} - Z_{t+1}^2 \,\big|\, \F_t\big] + \sigma^2(t+1) \\
&=\Phi(Q_t) + \sigma^2t = Y_t,
\end{align*}
where we use the facts that $\Exp{Z_t} = 0$ and $\Exp{Z_t^2} = \sigma^2$. Moreover, $Y_t$ has bounded increments, since $|Z_t| \leq Z_{\max} + Z_{\min}$ almost surely, and
\begin{align*}
    |Y_{t+1} - Y_t| & = |\Phi(Q_{t+1}) - \Phi(Q_t) + \sigma^2| \\
    & \leq |(M-2Q_t)Z_{t+1} - Z_{t+1}^2| + \sigma^2 \\
    & \leq M(Z_{\max} + Z_{\min}) + (Z_{\max} + Z_{\min})^2 + \sigma^2.
\end{align*}
Hence, by the optional stopping theorem:
\begin{align*}
    \Phi(S) = \Exp{\Phi(Q_{\tau(S)})} + \sigma^2 \E[\tau(S)] \implies \E[\tau(S)] = \frac{\Phi(S)-\Exp{\Phi(Q_{\tau(S)})}}{\sigma^2} \geq \frac{\Phi(S)}{\sigma^2} = \frac{S(M-S)}{\sigma^2},
\end{align*}
where the inequality follows from the fact that $Q_{\tau(S)} \not\in (0,M)$ by definition, and therefore $\Phi(Q_{\tau(S)}) \leq 0$.
\end{proof}

\subsection{Proof of \Cref{thm:Inefficient_for_fix_non-0_mean}}\label{apx:thm-3-proof}

\InefficiencyDrift*

\begin{proof}
We prove the fact for $\alpha < \frac{\Bmean}{\Nmean}$; $\alpha > \frac{\Bmean}{\Nmean}$ follows similarly. As before, for ease of notation we let $Z_t = Z_t(\alpha)$.  Note that $\mu = \Exp{Z_s} > 0$. 

We partition our analysis into two cases.

\smallskip 

 {\bf Case 1: There exist $\epsilon > 0$, $\delta > 0$ such that $\Pr(Z_t \geq \epsilon) \geq \delta$ and $\Pr(Z_t \leq - \epsilon) \geq \delta$.} By \Cref{Thm:Main_expression_for_inefficiency}, $H_M$ and $H_0$ exist as long as $E(M) < \infty$, $E(0) < \infty$, $p_M, p_0 \in (0,1)$. These conditions are satisfied in this case, by \Cref{lem:policy_constants}. Therefore, by \Cref{Thm:Main_expression_for_inefficiency} it suffices to bound $E(M)$, $E(0)$, $p_M$ and $p_0$.

We first exhibit upper bounds on $E(0)$ and $E(M)$. Consider the random variable $Q_t - \mu t$. It is easy to verify that $Q_t - \mu t$ is a martingale with bounded increments on the event $\{\tau > t\}$, since $|Q_t - Q_{t-1}| = |Z_t| \leq Z_{\max} + Z_{\min}$. Applying the optional stopping theorem to $Q_t - \mu t$, we have:
\begin{align*}
&\E[Q_\tau \mid S_0 = M] - \mu \E[\tau \mid S_0 = M] = M \\
\implies &E(M) := \E[\tau \mid S_0 = M] = \frac{\E[Q_\tau \mid S_0 = M]-M}{\mu} \leq \frac{Z_{\max}}{\mu},
\end{align*}
where the inequality uses the fact that $Q_{\tau} = Q_{\tau - 1}+Z_t \leq M + Z_{\max}$, by definition of $\tau$ and the fact that $Z_t \leq Z_{\max}$.

Similarly,
\begin{align*}
&\E[Q_\tau \mid S_0 = 0] - \mu\E[\tau \mid S_0 = 0] = 0 
\implies E(0) := \E[\tau \mid S_0 = 0] = \frac{\E[Q_\tau \mid S_0 = 0]}{\mu}.
\end{align*}
We have:
\begin{align*}
\E[Q_\tau \mid S_0 = 0] &= p_0 \E[Q_\tau \mid S_0 = 0, Q_\tau \leq 0] + (1-p_0)\E[Q_\tau \mid S_0 = 0, Q_\tau \geq M] \leq (1-p_0)\left(M+Z_{\max}\right).
\end{align*}
Hence, by \Cref{Thm:Main_expression_for_inefficiency},
\begin{equation}
\label{eq:h_m_bound_non_zero_mean}
H_M = \frac{1}{E(M) + E(0)\cdot \frac{1-p_M}{1-p_{0}}} \geq \frac{\mu}{ Z_{\max} + (M + Z_{\max}) \cdot(1-p_M)}.
\end{equation}

We now argue that $p_M = 1-O(M^{-1})$, which will establish that $H_M = \Omega(1)$. As in the proof of \Cref{thm:Inefficient_for_fix_0_mean}, we have:
\begin{align}\label{eq:pM-static-bad}
p_M &= \pr(Q_\tau \geq M \mid S_0 = M) \notag \\ &= \pr(Q_\tau \geq M \mid Z_1 < 0, S_0 = M){\Pr(Z_1 < 0)} + \pr(Q_\tau \geq M \mid Z_1 \geq 0,  S_0 = M)\pr(Z_1 \geq 0) \notag \\
    &= \pr(Q_\tau \geq M \mid Z_1 < 0, S_0 = M)\pr(Z_1 < 0) + \pr(Z_1 \geq 0) \notag \\
    &= 1-\Pr(Z_1 < 0) \left(1-\pr(Q_\tau \geq M \mid Z_1 < 0, S_0 = M)\right).
\end{align}

The following lemma provides a lower bound for $\pr(Q_\tau \geq M \mid Z_1 < 0, S_0 = M)$. We defer its proof to Appendix \ref{apx:static-bad}.

\begin{restatable}{lemma}{BoundStaticBad}\label{lem:static-bad}
\[\Pr(Q_\tau \geq M \mid Z_1 < 0, S_0 = M) \geq \frac{M + \E[Z_1 \mid Z_1 < 0]}{M+Z_{\max}}.\]
\end{restatable}

Plugging this bound into \eqref{eq:pM-static-bad}, we have:
\begin{align*}
p_M \geq 1-\Pr(Z_1 < 0)\left(1-\frac{M+\E[Z_1 \mid Z_1 < 0]}{M+Z_{\max}}\right) = 1-\Pr(Z_1 < 0)\left(\frac{Z_{\max}-\E[Z_1 \mid Z_1 < 0]}{M+Z_{\max}}\right) = 1-O(M^{-1}). 
\end{align*}

Having established that $H_M = \Omega(1)$ in this case, by \Cref{Cor:Main_expression_for_inefficiency}, $\overline{W} = \Omega(1)$. Using the fact that $\overline{W} = O(1)$ trivially, we obtain $\overline{W} = \Theta(1)$.

\medskip

\noindent {\bf Case 2: There do not exist $\epsilon > 0$, $\delta > 0$ such that $\Pr(Z_t \geq \epsilon) \geq \delta$ and $\Pr(Z_t \leq - \epsilon) \geq \delta$.} Since $\E[Z_t] > 0$, it must be that $\Pr(Z_t > 0) = 1$, with $\Pr(Z_t \geq \epsilon) \geq \delta$, for some $\epsilon > 0$, $\delta > 0$. By the same arguments as those used in the proof of \Cref{lem:policy_constants}, $\tau_M(S_0) < \infty$  almost surely. Moreover, once $S_t = M$, $S_{t'} = M$ for all $t' > t$, since $Z_t > 0$ almost surely. 
This gives us that:
\[
H_M = \limsup_{T \rightarrow \infty} \frac{1}{T} \sum_{t=1}^T \Ind{S_t = M} = \limsup_{T \rightarrow \infty} \frac{T - \tau_M(S_0)}{T} = 1.
\]
Thus we have $\overline{W} = \Theta(1)$ by  \Cref{Cor:Main_expression_for_inefficiency}.
\end{proof}

\subsubsection{Auxiliary Lemmas}\label{apx:static-bad}

\BoundStaticBad*

\begin{proof}
It is easy to see that, conditioned on $Z_1 < 0$, $Q_t$ is a submartingale. Hence, by the optional stopping theorem for submartingales:
\begin{align}\label{eq:opt-stop-submart}
\E[Q_\tau \mid Z_1 < 0, S_0 = M] \geq M + \E[Z_1 \mid Z_1 < 0].
\end{align}
Moreover:
\begin{align*}
\E[Q_\tau \mid Z_1 < 0, S_0 = M] &= \E[Q_\tau \mid Z_1 < 0, S_0 = M, Q_\tau \geq M] \Pr(Q_\tau \geq M \mid Z_1 < 0, S_0 = M)\\
&\quad + \E[Q_\tau \mid Z_1 < 0, S_0 = M, Q_\tau < M] \left(1-\Pr(Q_\tau {\geq} M \mid Z_1 < 0, S_0 = M)\right)
\end{align*}
Solving the above for $\Pr(Q_\tau \geq M \mid Z_1 < 0, S_0 = M)$, we obtain:
\begin{align*}
&\Pr(Q_\tau \geq M \mid Z_1 < 0, S_0 = M)\left(\E[Q_\tau \mid Z_1 < 0, S_0 = M, Q_\tau \geq M]-\E[Q_\tau \mid Z_1 < 0, S_0 = M, Q_\tau < M]\right)\\ &= \E[Q_\tau \mid Z_1 < 0, S_0 = M]-\E[Q_\tau \mid Z_1 < 0, S_0 = M, Q_\tau < M] \\
&\geq M + \E[Z_1 \mid Z_1 < 0] - \E[Q_\tau \mid Z_1 < 0, S_0 = M, Q_\tau < M],
\end{align*}
where the inequality follows from \eqref{eq:opt-stop-submart}. Using the fact that $\E[Q_\tau \mid Z_1 < 0, S_0 = M, Q_\tau \geq M] \leq M + Z_{\max}$, we have:
\begin{align*}
\Pr(Q_\tau \geq M \mid Z_1 < 0, S_0 = M) &\geq \frac{M + \E[Z_1 \mid Z_1 < 0] - \E[Q_\tau \mid Z_1 < 0, S_0 = M, Q_\tau < M]}{M+Z_{\max}-\E[Q_\tau \mid Z_1 < 0, S_0 = M, Q_\tau < M]}\\
&\geq \frac{M + \E[Z_1 \mid Z_1 < 0]}{M+Z_{\max}},
\end{align*}
since $\E[Q_\tau \mid Z_1 < 0, S_0 = M, Q_\tau < M] = \E[Q_\tau \mid Z_1 < 0, S_0 = M, Q_\tau \leq 0 ] \leq 0$ and the right-hand side of the above is decreasing in $\E[Q_\tau \mid Z_1 < 0, S_0 = M, Q_\tau < M]$.
\end{proof}


\section{Omitted Proofs from \Cref{ssec:Bang-Bang_Policy}}
\label{apx:bang-bang-full}

\subsection{\Cref{thm:bang-bang_inefficiency_bound}: Sufficient Conditions}\label{apx:delta-cond}

\begin{proposition}\label{prop:delta-cond}
Fix any $c > 0$ such that
\[\min\left\{\Pr\left(B_t \geq \Bmean+c\right), \Pr\left(B_t \leq \Bmean-c\right), \Pr\left(N_t \geq \Nmean+c/2\right), \Pr\left(N_t \leq \Nmean-c/2\right)\right\} > 0.\] Then, $0 < \Delta \leq \frac{c\left(\Bmean+\Nmean\right)}{\Nmean(\Nmean+c/2)}$ satisfies \cref{eq:thm-cond} in \Cref{thm:bang-bang_inefficiency_bound}.
\end{proposition}
 
\begin{proof}
Let $\delta = \min\left\{\Pr\left(B_t \geq \Bmean+c\right), \Pr\left(B_t \leq \Bmean-c\right), \Pr\left(N_t \geq \Nmean+c/2\right), \Pr\left(N_t \leq \Nmean-c/2\right)\right\} > 0$. Then, with probability at least $\delta^2$, we have:
\begin{align*}
B_t - N_t\left(\frac{\Bmean}{\Nmean}+\Delta/2\right) \geq \Bmean +c -(\Nmean-c/2)\left(\frac{\Bmean}{\Nmean}+\Delta/2\right) \geq c/2,
\end{align*}
for $\Delta \leq \frac{c\left(\Bmean+\Nmean\right)}{\Nmean(\Nmean+c/2)}$.
Similarly, with probability at least $\delta^2$:
\begin{align*}
B_t - N_t\left(\frac{\Bmean}{\Nmean}+\Delta/2\right) \leq \Bmean -c -(\Nmean+c/2)\left(\frac{\Bmean}{\Nmean}-\Delta/2\right) \geq -c/2,
\end{align*}
for $\Delta \leq \frac{c\left(\Bmean+\Nmean\right)}{\Nmean(\Nmean+c/2)}$.

Letting $\epsilon = c/2$ completes the proof of the claim.
\end{proof}

\subsection{Auxiliary Lemmas for Proof of \Cref{thm:bang-bang_inefficiency_bound}}
\label{apx:bang-bang-aux}

\HitProbabilityLemma*

\begin{proof}
We first show \cref{eq:hit-prob-from-M}. 
Conditioning on whether $Z_1 < 0$, we have:
\begin{align}\label{eq:first-lem}
\Pr\left(Q_{\Ttauhigh} < \frac{M}{2} \mid S_0 = M\right) &= \Pr\left(Q_{\Ttauhigh} < \frac{M}{2} \mid S_0 = M, Z_1 < 0\right) \Pr(Z_1 < 0 \mid S_0 = M) \notag \\ &\quad + \Pr\left(Q_{\Ttauhigh} < \frac{M}{2} \mid S_0 = M, Z_1 \geq 0\right)\left(1-\Pr(Z_1 < 0 \mid S_0 = M)\right) \notag \\
&= \Pr\left(Q_{\Ttauhigh} < \frac{M}{2} \mid S_0 = M, Z_1 < 0\right) \Pr(Z_1 < 0 \mid S_0 = M),
\end{align}
where the second equality follows from the fact that, if $Z_1 \geq 0$, $Q_1 \geq M$; therefore $\Ttauhigh = 1$ with $Q_{\Ttauhigh} \geq M$.

Consider now the event $\{Z_1 < 0\}$. Under the \ALG policy, $A_t = \frac{\Bmean}{\Nmean} + \Delta$ for all $s \leq  \Ttauhigh$. Therefore, \mbox{$\E[Z_s] = \Bmean - (\frac{\Bmean}{\Nmean} + \Delta)\Nmean =  -\Delta \cdot \Nmean$}, and $Q_t = \sum_{s=1}^t Z_s + M$ is a supermartingale for all $t \leq \Ttauhigh$. Therefore, by the optional stopping theorem for supermartingales,
\begin{align}\label{eq:opt-stop-supermart}
\E[Q_{\Ttauhigh} \mid Z_1 < 0, S_0 = M] \leq \E[Q_1 \mid Z_1 < 0, S_0 = M] = M + \E[Z_1 \mid Z_1 < 0, S_0 = M].
\end{align}
Moreover, by the law of total expectation,
\begin{align*}
\E[Q_{\Ttauhigh} \mid Z_1 < 0, S_0 = M] &= \E[Q_{\Ttauhigh} \mid Z_1 < 0, S_0 = M, Q_{\Ttauhigh} < M/2]\Pr\left(Q_{\Ttauhigh} < M/2 \mid Z_1 < 0, S_0 = M\right) \\& \quad+ \E[Q_{\Ttauhigh} \mid Z_1 < 0, S_0 = M, Q_{\Ttauhigh} \geq M/2]\left(1-\Pr\left(Q_{\Ttauhigh} < M/2 \mid Z_1 < 0, S_0 = M\right)\right) \\
\implies \Pr\left(Q_{\Ttauhigh} < M/2 \mid Z_1 < 0, S_0 = M\right) &= \frac{\E[Q_{\Ttauhigh} \mid Z_1 < 0, S_0 = M, Q_{\Ttauhigh} \geq M/2]-\E[Q_{\Ttauhigh} \mid Z_1 < 0, S_0 = M]}{\E[Q_{\Ttauhigh} \mid Z_1 < 0, S_0 = M, Q_{\Ttauhigh} \geq M/2]-\E[Q_{\Ttauhigh} \mid Z_1 < 0, S_0 = M, Q_{\Ttauhigh} < M/2]}.
\end{align*}

Observe that the event $Q_{\Ttauhigh} \geq M/2$ is equivalent to the event that $Q_{\Ttauhigh} \geq M$. Using the fact that the right-hand side of the above is increasing in $\E[Q_{\Ttauhigh} \mid Z_1 < 0, S_0 = M, Q_{\Ttauhigh} \geq M/2]$, and moreover $\E[Q_{\Ttauhigh} \mid Z_1 < 0, S_0 = M, Q_{\Ttauhigh} < M/2] \geq M/2-Z_{\min}$, we have:
\begin{align*}
\Pr\left(Q_{\Ttauhigh} < M/2 \mid Z_1 < 0, S_0 = M\right) &\geq \frac{M-\E[Q_{\Ttauhigh} \mid Z_1 < 0, S_0 = M]}{M-(M/2-Z_{\min})}\geq \frac{-\E[Z_1 \mid Z_1 < 0, S_0 = M]}{M/2+Z_{\min}},
\end{align*}
where the final inequality uses \eqref{eq:opt-stop-supermart}.

Plugging this back into \eqref{eq:first-lem}, we obtain:
\begin{align*}
\Pr\left(Q_{\Ttauhigh} < M/2 \mid S_0 = M\right) \geq \frac{\E[(Z_1)^- \mid S_0 = M]}{M/2+Z_{\min}} \geq C/M,
\end{align*}
for some $C > 0$, for large enough $M$, since $\E[(Z_1)^- \mid S_0 = M] \geq \epsilon \delta$. 

\medskip 

We now show \cref{eq:hit-prob-below}. Fix $S_0 = S \in [\frac{M}{2} - Z_{\min}, \frac{M}{2}]$. The following lemma will allow us to use the optional stopping theorem to bound $\Pr\left(Q_{\Ttaulow} \geq M/2 \mid S_0 = S\right)$. We defer its proof to the end of the section.

\begin{restatable}{lemma}{MeanNonZeroMartingale}  \label{lem:Mean_non-0_case_marginales}
    There exists a constant $C >0$ such that $e^{-C \Delta Q_t}$ is a super-martingale. Moreover, $e^{-C \Delta Q_t}$ has bounded increments on the event $\{\Ttaulow > t\}$.
\end{restatable}

Then, by \Cref{lem:Mean_non-0_case_marginales}, $\E\left[e^{-C\Delta Q_{\Ttaulow}} \mid S_0 = S\right] \leq e^{-C\Delta S} \leq e^{-C\Delta\left(M/2-Z_{\min}\right)}$. Applying the law of total expectation on $\E\left[e^{-C\Delta Q_{\Ttaulow}}\mid S_0 = S\right]$, we obtain:
\begin{align*}
e^{-C\Delta\left(M/2-Z_{\min}\right)} &\geq \E\left[e^{-C\Delta Q_{\Ttaulow}} \mid S_0 = S, Q_{\Ttaulow} \geq M/2\right]\Pr\left( Q_{\Ttaulow} \geq M/2 \mid S_0 = S\right)\\
&\quad + \E\left[e^{-C\Delta Q_{\Ttaulow}} \mid S_0 = S, Q_{\Ttaulow} \leq 0\right]\left(1-\Pr\left( Q_{\Ttaulow} \geq M/2 \mid S_0 = S\right)\right)\\
&\geq 1 - \Pr\left( Q_{\Ttaulow} \geq M/2 \mid S_0 = S\right),
\end{align*}
where we use the fact that $\E\left[e^{-C\Delta Q_{\Ttaulow}} \mid S_0 = S, Q_{\Ttaulow} \leq 0\right] \geq 1$ for the second inequality. Solving for $\Pr\left( Q_{\Ttaulow} \geq M/2 \mid S_0 = S\right)$, we obtain:
\begin{align*}
\Pr\left( Q_{\Ttaulow} \geq M/2 \mid S_0 = S\right) \geq 1-e^{-C\Delta (M/2-Z_{\min})}.
\end{align*}
Noting that $Z_{\min} \leq M/4$ for large enough $M$, and setting $C_1 = C/4$, we obtain the result.

Finally, \cref{eq:hit-prob-above} follows almost identical lines as the derivation of \cref{eq:hit-prob-below}. In particular, one can show that there exists $C > 0$ such that $e^{C\Delta Q_t}$ is a supermartingale. Then, applying the optional stopping theorem in the same way, the lower bound on $\Pr\left(Q_{\Ttauhigh} < M/2 \mid S_0 = S\right)$ emerges. We omit this argument for brevity.
\end{proof}

\MeanNonZeroMartingale*
\begin{proof}
Fix $C > 0$, and let $Y_t = e^{-C \Delta \cdot Q_t}$. For ease of notation, we let $\lambda = C \Delta$. We have:
    \begin{align*}
    \E[Y_t-Y_{t-1} \mid \mathcal{F}_{t-1}] &= \E\left[e^{-\lambda Q_t}-e^{-\lambda Q_{t-1}} \mid \mathcal{F}_{t-1}\right] \\
    &= \E\left[e^{-\lambda Q_{t-1}}\left(e^{-\lambda \left(B_t-\left(\alpha^*-\Delta\right)N_t\right)}-1\right) \mid \mathcal{F}_{t-1}\right]\\
    &= Y_t\E\left[e^{-\lambda \left(B_t-\left(\alpha^*-\Delta\right)N_t\right)}-1\right],
    \end{align*}
    where the second equality uses the fact that, given $S_0 \in [M/2-Z_{\min}, M/2]$, for all $t < \Ttaulow$, $A_t = \alpha^*-\Delta$.

    We now argue that $\E[e^{-\lambda \left(B_t-\left(\alpha^*-\Delta\right)N_t\right)}] < 1$ for some $C > 0$, and therefore $\E[Y_t-Y_{t-1} \mid \mathcal{F}_{t-1}] \leq 0$. For ease of notation, let $X_t = B_t-\left(\alpha^*-\Delta\right)N_t$, and $\hat{X}_t = X_t-\E[X_t] = B_t-\left(\alpha^*-\Delta\right)N_t-\Delta\mu_{\NDist}$. Since $B_t$ and $N_t$ are both bounded, $\hat{X}_t \in \left[-\alpha^*N_{\max}-\Delta\mu_{\NDist},B_{\max}-\Delta\mu_{\NDist}\right]$ is sub-Gaussian. Then, by Hoeffding's lemma, for any $s \in \mathbb{R}$:
    \begin{align*}
    &\E[e^{s\hat{X}_t}] \leq \exp\left(\frac{s^2\left(\left(B_{\max}-\Delta\mu_{\NDist}\right)-\left(-\alpha^*N_{\max}-\Delta\mu_{\NDist}\right)\right)^2}{8}\right)=\exp\left(\frac{s^2\left(B_{\max}+\alpha^*N_{\max}\right)^2}{8}\right)\\
    \implies &\E[e^{sX_t}] \leq \exp\left(\frac{s^2\left(B_{\max}+\alpha^*N_{\max}\right)^2}{8}+s\E[X_t]\right) = \exp\left(\frac{s^2\left(B_{\max}+\alpha^*N_{\max}\right)^2}{8}+s\cdot \Delta\mu_{\NDist}\right).
    \end{align*}
    Letting $s = -C\Delta$, it suffices to show that there exists $C > 0$ such that 
    \begin{align*}
\exp\left(\frac{(C\Delta)^2\left(B_{\max}+\alpha^*N_{\max}\right)^2}{8}-C\Delta^2\mu_{\NDist}\right) < 1 
&\iff \frac{(C\Delta)^2\left(B_{\max}+\alpha^*N_{\max}\right)^2}{8}-C\Delta^2\mu_{\NDist} < 0 \\
&\iff C\cdot \frac{\left(B_{\max}+\alpha^*N_{\max}\right)^2}{8}-\mu_{\NDist} < 0 \\
&\iff C < \frac{8\mu_{\NDist}}{\left(B_{\max}+\alpha^*N_{\max}\right)^2}.
    \end{align*}

    We conclude by arguing that $Y_t$ has bounded increments on the event $\Ttaulow > t$. Note that, for $t < \Ttaulow$, $Q_t \in (0,M/2)$. Therefore, by the mean value theorem,
    \begin{align*}
    |e^{-C\Delta Q_t}-e^{-C\Delta Q_{t-1}}| \leq \max_{c\in (0,M/2)} C\Delta e^{-C\Delta c}  |Q_t-Q_{t-1}| \leq C\Delta\max\{Z_{\max},Z_{\min}\} < \infty.
    \end{align*}
\end{proof}

\section{Extensions}
\label{app:extensions}

\subsection{Proof of \Cref{thm:extensions}}\label{apx:extensions-proof}

\ExtensionThm*
\begin{proof}
The result follows directly from \cref{thm:bang-bang_inefficiency_bound}.  Indeed, recall that we split the warehouse into $K$ virtual stores of capacity $M/K$ and run the same \ALG policy independently on each virtual store. For any sample path, the global inefficiency in the system (whether overflow or stockouts) is no larger than the sum of the inefficiencies accrued inside the virtual stores. Hence we have that $\overline{W} \leq \sum_k \overline{W}_k$ where:
\[
\overline{W}_k = \lim \sup_{T \rightarrow \infty} \sum_{t=1}^T W_{t,k},
\]
(and similarly for $\overline{V}$).  {This reduction is conservative, but sufficient: each virtual store is now a one-dimensional system with capacity $M / K = \Theta(M)$, budget process $B_{t,k}$, and demand process $\sum_\theta N_{t,\theta}$, to which the single-resource analysis from \cref{thm:bang-bang_inefficiency_bound} applies.}  Therefore, $\overline{W}_k = e^{-\Omega(\Delta M)}$, and as a result, $\overline{W} = e^{- \Omega(\Delta M)}$. Similarly for $\overline{V}$.

The result for $\Delta = 0$ follows from an identical argument and applying \cref{thm:Inefficient_for_fix_0_mean}.
\end{proof}

\subsection{Extending \ALG to Type-Dependent Allocations}\label{apx:extensions-improvement}

We close out with showing how to incorporate type-dependent preferences in the baseline allocations. 

We define the {\em fluid Eisenberg-Gale} (EG) allocation  $\{a_\theta^{\EG}\}_{\theta\in\Theta}$ as the solution to
\begin{equation}
\label{eq:EG}
\begin{aligned}
\max_{\{a_\theta\}_{\theta\in\Theta}} \quad & \sum_{\theta\in\Theta} \mu_{\mathcal{N}, \theta} \,\log\bigl(w_\theta^\top a_\theta\bigr) \\
\text{s.t.} \quad & \sum_{\theta\in\Theta}\mu_{\mathcal{N},\theta}\, a_\theta \;\le\; \Bmean, \\
& a_\theta \;\ge\; 0,\qquad \forall \theta\in\Theta.
\end{aligned}
\end{equation}

\cref{eq:EG} represents a fluid-limit version of the classical Eisenberg-Gale program, which is known to produce allocations that are Pareto-efficient, envy-free, and proportional in offline resource allocation settings~\citep{eisenberg1961aggregation,varian1973equity,jain2010eisenberg}. In our context, we use \cref{eq:EG} to define {\em type-specific} allocations $a_{\theta,k}^{\EG}$, in contrast to the earlier baseline of $\mu_{\mathcal{B},k} / \mu_{\mathcal{N}}$, which does not account for differences in type-specific weights.

\paragraph{Modified Bang-Bang policy over virtual stores.}
Partition the warehouse into $K$ virtual stores, each with capacity $M/K$. Fix a fairness budget $\Delta>0$.
Let $S_{t-1,k}$ denote the inventory in virtual store $k$ at the start of period $t$. For every arriving type $\theta$,
\begin{equation}
\label{eq:bb-eg}
A_{t,\theta,k} \;=\;
\begin{cases}
a_{\theta,k}^{\EG} - \frac{\Delta}{2} & \text{if } S_{t-1,k} < \dfrac{M}{2K},\\[2mm]
a_{\theta,k}^{\EG} + \frac{\Delta}{2} & \text{if } S_{t-1,k} \ge \dfrac{M}{2K}.
\end{cases}
\end{equation}

It is again easy to see that the resulting algorithm satisfies $\Dfair = \Theta(\Delta)$.  Additionally, one can also show it achieves {\em counterfactual envy} (again first considered in \citet{sinclair2021sequential} as a ``regret'' benchmark against the fluid EG solution) as:
\[
\max_{t,\theta} |u(A_{t,\theta}, \theta) - u(a_{\theta}^{\EG}, \theta)| = \Theta(\Delta).
\]
We conjecture that a straightforward (albeit technical) modification to our analysis (by adjusting the `center' points in the technical results around $a_{\theta,k}^{\EG}$ vs $\Bmean / \Nmean$) shows the envy-inefficiency bounds for this modified algorithm.

\section{Numerical Experiments: Additional Details}
\label[appendix]{app:simulation_details}

\subsection{Simulation Details}
\label[appendix]{sec:experiment_details}
\paragraph{Computing Infrastructure:} The experiments were conducted on a personal computer with an Apple M2, 8-core processor and 16.0GB of RAM.

\paragraph{Experiment Setup:} Each experiment was run with $100$ iterations where the various plots and metrics were computed with respect to the mean of the various quantities.  To evaluate the performance of the algorithms on $\Deff$ we use a time horizon of $T = 10,000$ and evaluate:
\[
\Deff = \frac{1}{T} \sum_{\tau < t} h W_t + b V_t.
\]
To avoid errors in numerical precision when plotting $\log(\Deff)$ we cap all values less than $10^{-4}$ to be $10^{-4}$.

\subsection{Additional Results with Multiple Resources}
\label[appendix]{app:experiments_multi_resources}

To complement our simulation results in \cref{sec:simulations}, we include additional simulations for the model in \cref{sec:extensions} (and discussed in further details in \cref{app:extensions}) with multiple resources, multiple types of customers, and a single warehouse with a fixed capacity constraint $M$.

\begin{table}[!t]
\caption{Weights $w_k$ for the different products considered in the \textbf{Multi-FBST} experiments.  Here we use the weights taken from the historical prices used in the market mechanism to distribute food resources to food pantries across the United States~\citep{prendergast2017food}.}
\label{tab:weights}
\setlength\tabcolsep{0pt} 
\centering
\begin{tabular*}{\columnwidth }{@{\extracolsep{\fill}}rccccc}
\toprule
Resource & Cereal &  Pasta & Prepared Meals & Rice & Meat\\
\midrule
Weights (type $\theta = $ omnivore) & 3.9 & 3 & 2.8 & 2.7 & 1.9\\
Weights (type $\theta = $vegetarian) & 3.9 & 3 & .1 & 2.7 & .1\\
Weights (type $\theta = $prepared-only) & 3.9 & 3 & 2.8 & 2.7 & .1 \\
\bottomrule
\end{tabular*}
\end{table}

\begin{figure}[t]
\centering
\begin{subfigure}[t]{0.48\linewidth}
    \centering
    \includegraphics[width=\textwidth]{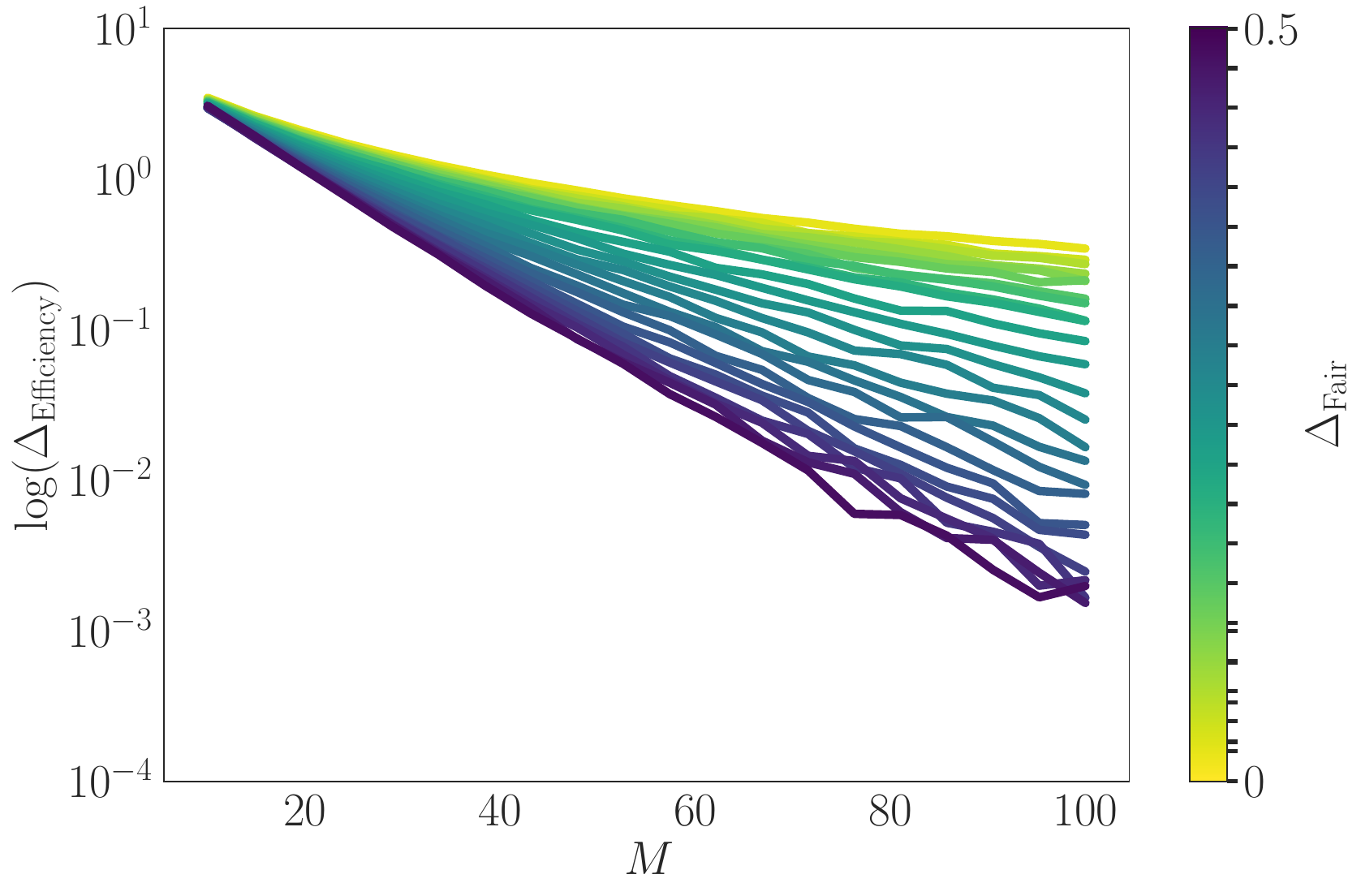}
    \caption{\small $\Deff$ vs $M$}
    \label{fig:multi_M}
\end{subfigure}
\hfill
\begin{subfigure}[t]{0.48\linewidth}
    \centering
    \includegraphics[width=\textwidth]{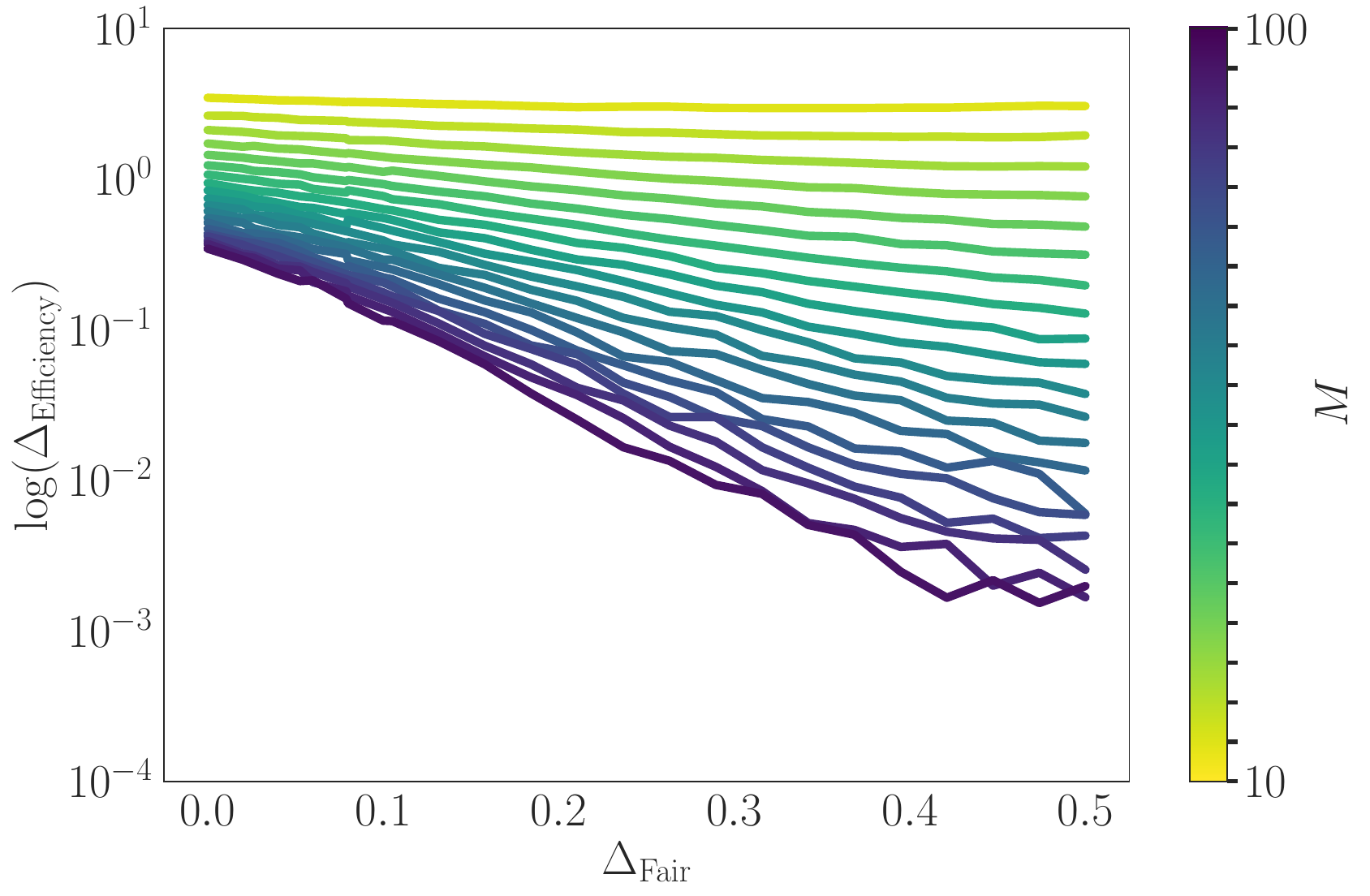}
    \caption{\small $\Deff$ vs $\Dfair$}
    \label{fig:multi_Delta}
\end{subfigure}
\caption{{Envy-inefficiency scaling with multi-resources and multi-types}: As before, we compare inefficiency on a log scale ($\log(\Deff)$, y-axis) against $\Dfair$ (y-axis) or $M$ (y-axis).  Each curve represents the performance of $\ALG(\Dfair)$ for $M \in [10, 100]$ and $\Dfair \in [0, 0.5]$.  Darker curves correspond to larger values of $M$ (\cref{fig:time_vary_delta}) or $\Dfair$ (\cref{fig:time_varying_M}). The exponential scaling, and sharp phase transition near $\Dfair=0$ are still clearly visible.}
\label{fig:multi_resource}
\vspace{-0.5em}
\end{figure}

\paragraph{Experimental Setup.} When evaluating $\Deff$ we consider the cost $h$ for overflows and $b$ for stockouts as $h = b = 1$ (results for other values of $(h,b)$ are similar and omitted as such). We test twenty different capacity constraints $M$ equally spaced in $[10,100]$ and twenty five different $\Dfair$ values equally spaced in $[0,0.5]$.

We consider the setting of five resources $K$ (corresponding to cereal, pasta, prepared meals, rice, and meat) and three types $\Theta$ (corresponding to vegetarians, omnivores, and ``prepared-food only'' individuals).  To estimate the utility functions for the different types we use the historical prices $p_k$ used in the market mechanism to distribute food resources to food pantries across the United States \citep{prendergast2017food}.  In particular, for each of the types $\theta$ we set the weight $w_{\theta, k} = p_k \Ind{\text{type } \theta \text{ uses resource } k}$, with an additional weight of $0.1$ for the latent carry-over need.  For example, the indicator is zero for the vegetarian type and the meat resource (see \cref{tab:weights} for the full table of weights).

We also follow the {\bf Truncated Normal} simulation set-up, and assume $B_{t,k} \sim \Norm(\mu_{\BDist, k}, \sigma_n^2)^+$ and $N_t \sim \Norm(\mu_{\NDist, \theta}, \sigma_b^2)^+$ for $\sigma_b = \sigma_n = 1$ and $\mu_{\BDist, k} = 5$ for all $k$, and $\mu_{\NDist} = [1.25, 1.5, 2.25]$ to reflect the portion of estimated customers with those food preferences.  We note that this is a stochastic budget replenishment extension of the simulations studied in \citet{sinclair2021sequential}.

\paragraph{Summary} The multi-resource, multi-type simulations exhibit the same qualitative envy–inefficiency trade-off as in the single-resource, single-type model.  
In \cref{fig:multi_resource} (left) we compare the performance of $\Deff$ as we vary the storage capacity $M$.  For small values of $M$, inefficiency remains high across all values of $\Dfair$ due to frequent stockouts and overflow costs. However, as $M$ increases, we observe exponential improvement in efficiency.
In \cref{fig:multi_resource} (right) we compare the performance of $\Deff$ as we vary the fairness parameter $\Dfair$.  We again see performance improvements for $\Deff$ under larger values of $\Dfair$, while holding the capacity constraints $M$ fixed.  Moreover, even with a slight relaxation of fairness ($\Dfair > 0$) yields significant efficiency improvements.  We note that on the whole, the performance of $\Deff$ is worse than in the single-resource setting, due to the fact that the algorithm is run via virtual-stores, so with the same capacity $M$ the ``effective'' capacity is reduced by $M/K$ to account for the multiple resources.

\end{APPENDICES}







\end{document}